\documentclass[]{iopart}
\usepackage{hyperref}
\expandafter\let\csname equation*\endcsname\relax

\expandafter\let\csname endequation*\endcsname\relax
\usepackage{amsfonts,amsmath,
  amstext,amsbsy,amsopn,amsthm,amscd}
\usepackage{enumerate}

\newtheorem{theorem}{Theorem}[section]
\newtheorem{lemma}[theorem]{Lemma}
\newtheorem{corollary}[theorem]{Corollary}
\newtheorem{proposition}[theorem]{Proposition}

\theoremstyle{definition}

\newtheorem{definition}[theorem]{Definition}
\newtheorem{assu}[theorem]{Assumption}
\newtheorem{rem}[theorem]{Remark}

\theoremstyle{remark}

\eqnobysec 
\newenvironment{proofof}[1]{\emph{Proof of #1.}}{\hfill$\square$}

\newcommand{\field}[1]{\mathbb{#1}}

\newcommand{\R}{\field{R}}

\DeclareMathOperator{\dom}{dom}    
\DeclareMathOperator{\dive}{div}   
\DeclareMathOperator{\diag}{diag}   

\newcommand{\embeds}{\hookrightarrow}

\renewcommand{\O}{\Omega}

\renewcommand{\mr}{\mathrm{MR}}

\renewcommand{\S}{\widehat{R}}
\newcommand{\sol}{\mathcal{S}}
\newcommand{\afo}{\kappa} 
\newcommand{\bfo}{\sigma} 
\newcommand{\cfo}{\xi} 
\newcommand{\vecL}{\mathbb{L}}
\newcommand{\domLap}{\dom_{L^p}(\Delta)}
\newcommand{\domVecLap}{\dom_{\vecL^p}(\Delta)}

\newcommand{\iv}[1]{{\mathopen]#1\mathclose[}}
\newcommand{\dd}{\,\mathrm{d}}
\newcommand{\x}{\mathrm{x}}
\newcommand{\sg}[1]{\mathop{\exp(#1)}}
\newcommand{\sgb}[1]{\mathop{\exp\bigl(#1\bigr)}}
\newcommand{\LL}{\mathcal{L}}
\renewcommand{\BB}{\mathcal{B}}

\usepackage{color}

\begin{document}
\title{The full Keller-Segel model is well-posed on nonsmooth domains}
\author{D Horstmann$^1$, H Meinlschmidt$^2$ and J Rehberg$^3$} \address{$^1$ Mathematisches Institut der
  Universit\"at zu K\"oln, Weyertal 86--90, D-50931 K\"oln, Germany} \ead{dhorst@math.uni-koeln.de}
\address{$^2$ Faculty of Mathematics, TU Darmstadt, Dolivostr.~15, D-64293 Darmstadt, Germany}
\ead{meinlschmidt@mathematik.tu-darmstadt.de} \address {$^3$ Weierstrass Institute for Applied Analysis
  and Stochastics, Mohrenstr.~39, D-10117 Berlin, Germany} \ead{rehberg@wias-berlin.de}
\begin{abstract}
  In this paper we prove that the full Keller-Segel system, a quasilinear strongly
  coupled reaction-crossdiffusion system of four parabolic equations, is well-posed in space dimensions $2$ and $3$ in
  the sense that it always admits an unique local-in-time solution in an adequate function space,
  provided that the initial values are suitably regular. The proof is done via an abstract solution
  theorem for nonlocal quasilinear equations by Amann and is carried out for general source terms. It is
  fundamentally based on recent nontrivial elliptic and parabolic regularity results which hold true even
  on rather general nonsmooth spatial domains. This enables us to work in a nonsmooth setting which is
  not available in classical parabolic systems theory. Apparently, there exists no comparable existence
  result for the full Keller-Segel system up to now. Due to the large class of possibly nonsmooth domains
  admitted, we also obtain new results for the ``standard'' Keller-Segel system consisting of only two
  equations as a special case.
\end{abstract}

\maketitle



\section[Introduction] {Introduction}
%
This paper establishes the local-in-time existence of solutions in a suitable functional-analytic sense
to the so-called original \emph{full Keller-Segel model} which is a coupled system of four nonlinear
parabolic partial differential equations over a finite time horizon $J = \iv{0,T}$ in a bounded domain
$\Omega \subset \R^d$ in space dimensions $d \in \{2,3\}$, and reads as follows:
\begin{align}
  \label{e-u}
  u'- \dive\left( \afo(u,v) \nabla u\right) & = \dive\left( \bfo(u,v) \nabla v\right) & \text{in } J \times \O, \\
  \label{e-v}
  v'- k_v \Delta v & = -r_1vp +r_{-1} w +uf(v) & \text{in } J \times \O,\\
  p'- k_p \Delta p & = - r_1vp +(r_{-1}+r_2)w +u g(v,p) & \text{in } J \times \O,\\
  \label{e-w}
  w'- k_w \Delta w & =  r_1vp -(r_{-1}+r_2)w & \text{in } J \times \O,
\end{align}
combined with homogeneous Neumann conditions
%
\begin{align}
  \label{e-neumann}
  \nu\cdot\kappa(u,v)\nabla u = \nu\cdot k_v\nabla v & =  \nu\cdot k_p\nabla p =
  \nu\cdot k_w\nabla w = 0 & \text{on } J \times \partial\O,
\end{align}
where $\nu$ denotes the outer unit normal to the boundary $\partial\Omega$, and suitable initial values
\begin{align}
  \label{e-IVs}
  \bigl(u(0,\cdot),v(0,\cdot),p(0,\cdot),w(0,\cdot)\bigr) & = (u_0,v_0,p_0,w_0) & \text{in } \O.
\end{align}
Before we elaborate on the origin and biological meaning of this model, let us explain a critical
property of this system of parabolic equations. The coefficient function $\sigma$ in~\eqref{e-u} is
\emph{not} assumed to be definite in sign and generally not restricted in its magnitude. This implies
that the spatial second order system differential operator underlying~\eqref{e-u}--\eqref{e-w} fails to
satisfy the usual strong ellipticity conditions in the form of the Legendre-- or Legendre-Hadamard
conditions; in particular, there is in general no G{\aa}rding inequality available,
cf.~\cite{GiaMarti,Zhang}. The system \emph{is} normally elliptic in the sense of
Amann~\cite{Am90}---also known as Petrowskii parabolic~\cite[Ch.~VII.8]{ladypara}---and as such admits
existence of local-in-time solutions quite immediately under the assumptions there.  These assumptions
however include ``smoothness'', or at least $C^2$-regularity, of the boundary $\partial\Omega$ and it is
not clear how to adopt the theory to less smooth situations. Under such smoothness assumptions, the results
in~\cite{Am90} have been used already to obtain local-in-time existence of a related system,
cf.~\eqref{simplified} below, for instance in~\cite{Cie07,taowinkler,Wrz04}. Let us note that the
authors in~\cite{Gajewski:1998} deal with a related system in a piecewise~$C^2$-setting. 

It is the aim of this work to show the existence of local-in-time solutions of~\eqref{e-u}--\eqref{e-IVs}
in a generally \emph{nonsmooth} setting for
$\Omega$, namely that of a Lipschitz domain. Since we, as explained, cannot use established theory for
parabolic systems, the strategy for our proof is to solve the lower three equations for
$(v,p,w)$ in dependence of the function $u$ and to re-insert this dependence for
$v$ in the first equation. This way, we obtain a \emph{single}, albeit quite involved, parabolic equation
for
$u$ for which we can rely on the full power of recent elliptic~(\cite{auscher,disser,Ziegler}) and
parabolic~(\cite{GKR,HiebRehb}) results, available for very general geometric constellations, in order to
treat it, thereby using a fundamental theorem by Amann~\cite[Thm.~2.1]{AmannDiffEq}. Following this
strategy, we also obtain new results for the related system mentioned above, the classical two-equation
Keller-Segel model of chemotaxis (\eqref{simplified} below), and similar systems in a nonsmooth setting.

The consideration of a nonsmooth boundary for $\Omega$ is not an academic example but motivated by
observations from numerical simulations of both~\eqref{e-u}--\eqref{e-IVs} and simplified models. For
instance, these numerical simulations show a concentration behavior of the solution in the smallest
interior angle of the considered domain. There is also a connection between the geometry of the domain
and the precise critical mass that insures the global-in-time existence of a solution on nonsmooth
domains, given as a multiple of the smallest interior angle of the domain (see for
example~\cite[Thm.~4.3, Rem.~4.5]{Gajewski:1998}). In this sense, it is of interest to establish
(local-in-time) existence results also for a generally nonsmooth boundary of $\Omega$.

\subsection{Biological background}

The above model describes the aggregation phase during the life cycle of cellular slime molds like the
\emph{Dictyostelium discoideum} and was first introduced by Keller and Segel in their 1970ies
paper~\cite{KellerSegel}. We briefly describe the underlying biological processes.  Looking at its life
cycle one observes that a myxamoebae population of the Dictyostelium grows by cell division as long as
there are enough food resources. When these are depleted, the myxamoebae propagate over the entire domain
available to them.  Then, after a while, the phase that is covered by the given model is initiated by one
cell that starts to exude cyclic Adenosine Monophosphate (cAMP) which attracts the other myxamoebae. As a
consequence the other myxamoebae are stimulated to move in direction of the so-called founder cell and
commence to release cAMP. This leads to the aggregation of the myxamoebae that also start to
differentiate within the myxamoebae aggregates resp. within the aggregation centers. The aggregation
phase ends with the formation of a pseudoplasmoid in which every myxamoebae maintains its individual
integrity. However, Keller and Segel did not model the formation of the pseudoplamoid; thus, this phase
of the life cycle of the Dictyostelium is \emph{not} covered in the original equations. This
pseudoplasmoid is attracted by light and, therefore, it moves towards light sources. Finally a fruiting
body is formed and after some time spores are diffused from which the life cycle begins again.  For more
details on the life cycle of the Dictyostelium we refer to~\cite{Bonner:1967}, for example.
%

%
In the given model $u(t,x)$ denotes the myxamoebae density of the cellular slime molds at time $t$ in
point $x$, where $v(t,x)$ describes a chemo-attractant concentration (like cAMP).  The given model for
aggregation of a cellular slime population is based on four basic processes that can be observed during
the aggregation phase:
\begin{enumerate}[a)]
\item The chemo-attractant is produced per amoeba at a positive rate $f(v)$.
\item The chemo-attractant is degraded by an extra-cellular enzyme, where the concentration of the is
  enzyme at time $t$ in point $x$ is denoted by $p(t,x)$.  This enzyme is produced by the myxamoebae at a
  positive rate $g(v,p)$ per amoeba.
\item Following Michaelis-Menten the chemo-attractant and the enzyme react to form a complex ${\cal E}$ of
  concentration $w$ which dissociates into a free enzyme plus the degraded product:
  \begin{equation*}
    v+p\quad
    \genfrac{}{}{0pt}{0}{\genfrac{}{}{0pt}{1}{r_1}{\longrightarrow}}{\genfrac{}{}{0pt}{1}{\longleftarrow}{r_{-1}}}
    \quad\mathcal{E} \quad
    \genfrac{}{}{0pt}{1}{r_2}{\longrightarrow}
    \quad p \,+\text{ degraded product,}
  \end{equation*}
  where $r_{-1},~r_1$ and $r_2$ are positive constants representing the reaction rates.
\item The chemo-attractant, the enzyme and the complex diffuse according to Fick's law.
\end{enumerate}
As a tribute to the experimental setting and the conservation of the myxamoebae density the equations are
equipped with homogeneous Neumann boundary data.
%

Since the influence of chemical substances in the environment on the movement of motile species (in
general called chemotaxis) can lead to strictly oriented or to partially oriented and partially tumbling
movement of the species, the first equation contains both a pure diffusion term
$\dive\left(\afo(u,v) \nabla u\right)$ with $\afo(u,v)\geq 0$ for nonnegative functions $(u,v)$, and a
convection term $\dive\left( \bfo(u,v) \nabla v\right)$ which describes the movement with respect to the
chemical concentration.  For a movement towards a higher concentration of the chemical substance, termed
\emph{positive} chemotaxis, one assumes $\bfo(u,v)<0$ for nonnegative $(u,v)$, while for the movement
towards regions of lower chemical concentration, called \emph{negative} chemotactical movement, the
opposite inequality $\bfo(u,v)>0$ has to hold. For the detailed derivation of the given model we refer
to~\cite{Horstmann1,KellerSegel}.
%

Chemotaxis is known to be an important device for cellular communication. In development or in living
tissues the communication by chemical signals prearranges how cells collocate and organize themselves.
Biologists studying chemotaxis often concentrate their experiments on the movement, the self-organization
and pattern formations of the cellular slime mold Dictyostelium discoideum.  One reason for the great
interest in this cellular slime mold is caused by the fact that ``\emph{development in Dictyostelium
  discoideum results only in two terminal cell types, but processes of morphogenesis and pattern
  formation occur as in many higher organisms}'' (see~\cite[p.~354]{Nanjundiah:1992}).  Thus biologists
hope that studying this cellular slime mold gives more insights in understanding cell differentiation.
%
\subsection{Context and related work}
By to a simplification done by Keller and Segel themselves in~\cite{KellerSegel}, the original model of
four strongly coupled parabolic equations~\eqref{e-u}--\eqref{e-w} was reduced to a model which is given
by a system of only \emph{two} strongly coupled parabolic equations. This was done by assuming that the
complex is in a steady state with regard to the chemical reaction and that the total concentration of the
free and the bounded enzyme is a constant; assumptions that are well-known for the Michaelis-Menten
equations in enzyme kinetics. The reduction was justified by the paradigm that ``\emph{it is useful for
  the sake of clarity to employ the simplest reasonable model}'' (see~\cite[p.~403]{KellerSegel}). The
corresponding model is then given by the following parabolic equations:
\begin{equation}
  \left.\begin{aligned}
      u' - \dive\left(\afo(u,v)\nabla u \right)& = \dive\left( \bfo(u,v)\nabla v\right)  & \text{in $J \times \Omega$}, \\
      v_t - k_c \Delta v & = -k(v)v + uf(v)\qquad  & \text{in $J \times \Omega$}, \\ \nu\cdot\afo(u,v)\nabla u & = \nu\cdot k_c\nabla v = 0 & \text{on $J \times \partial\Omega$}, \\
      \bigl(u(0,\cdot),v(0,\cdot)\bigr) & = (u_0,v_0) & \text{in $\Omega$}.
    \end{aligned}
    \qquad \right\}\label{simplified}
\end{equation}
This model is nowadays often referred to as the classical chemotaxis model or as the Keller-Segel model
in chemotaxis. As in the full model, $\kappa(u,v)$ denotes the density-dependent diffusion coefficient
and $\sigma(u,v)$ is the chemotactic sensitivity, where now $k(v)v$ and $uf(v)$ describe degradation and
production of the chemical signal.  For $\kappa(u,v)=1$, $\sigma(u,v)= - \chi\cdot u$ or
$ - \chi\frac{u}{v}$ with a constant $\chi>0$ and $k(\cdot)$ and $f(\cdot)$ positive
constants, 
this two-equation model has been extensively studied during the last twenty years, see for
instance~\cite{Hillen:Painter:2002,Hillen:Painter,Horstmann1,Horstmann2,Horstmann3} and the references
therein. In particular the so-called Childress-Percus conjecture~\cite{Childress:1981}
for~\eqref{simplified} concerning $L^\infty$ blow-up behavior has attracted many scientists. Subdividing
via space dimension we mention~\cite{Osaki:2001} for $d=1$ and, among
others,~\cite{Biler:1998,Gajewski:1998,Herrero1,Horstmann4,Horstmann5,Horstmann6,Nagai1} for $d = 2$, as
well as~\cite{Boy:1996,Perthame1,Horstmann7,Horstmann8,Winkler1} for $d = 3$.
%

%

From the biological point of view, the blow-up behavior of the solution can be interpreted as the
starting point of cell differentiation and therefore the blow-up time $T_{\max}<\infty$ would correspond
to the stopping time where the aggregation phase in the life cycle of the Dictyostelium ends and the cell
differentiation and formation of the pseudoplasmoid starts.
%
%
%

Besides the mathematical interesting question whether the solution can blow up in finite or in infinite
time one can also observe interesting pattern formations during the aggregation phase and development of
the Dictyostelium such as traveling waves like motion and spiral waves for the chemo-attractant. Although
there have been some attempts to prove the existence of traveling wave solutions and to simulate
sunflower spirals for the simplified model~\eqref{simplified}---see for
instance~\cite{Aotani,Horstmann2,Horstmann9,Wang} and the references therein---, one seems to need more
complicated chemotaxis systems consisting of more than only two equations to describe such kind of
pattern formation.  However, these more complicated systems still fit in the general setting of the full
Keller-Segel model as considered in the present paper (cf.~\eqref{e-preu}--\eqref{e-anfang} on
page~\pageref{e-preu} below). Hence, it might be worthwhile to work on the original four-equation-system
instead if one tries to describe these pattern formations during the aggregation of some particular
species. Possibly, the reduction to two equations that was done in~\cite{KellerSegel} was too restrictive
to cover all observable patterns and phenomena during the aggregation of mobile species like the
Dictyostelium discoideum. As another example, one can find an attempt to describe the aggregation of the
Dictyostelium discoideum along the experimentally observable cAMP spiral waves in~\cite{Vasiev-Panfilov}
where the authors consider a coupled three-equations model that contains a version of the simplified
Keller-Segel model complemented with an ODE that covers the recovery process of the myxamoebae after
binding the extra-cellular cAMP. As above, it seems worthwhile to investigate the original full model to
see whether it can also generate these complex pattern formations.
%

As far as we know there are no results available for the full four-equation model on nonsmooth
domains. In particular, the question of blow-up has, as far as we know, not been studied for the full
four equations model up to now. Of course, there are several local-existence results
known for parabolic-parabolic and parabolic-elliptic versions of the simplified two equation
model~\eqref{simplified} as for instance the results
in~\cite{Alt:Habil,Biler:1998,BilerZienkiewicz,Hillen:Painter:2002,Nagai:1997,
  Senba:Suzuki:Applied-Analysis2004,Stinner1,Yagi}. Furthermore, existence results for solutions for the
simplified two-equation model with additional population growth are also known,
cf.~\cite{Kang1,Osaki:2002,Tello1,Winkler2,Xiang1}. Some of these results may be extended to the full
model~\eqref{e-u}--\eqref{e-IVs}; however, all of them consider the equation either on a smooth domain
with boundary of class $C^2$, on convex domains with smooth boundaries, or on the whole space
$\R^d$. Furthermore, the initial data has to satisfy certain comparability conditions in some cases.  The
only result which we are aware of concerning nonsmooth objects is the local existence result
in~\cite{Gajewski:1998} where the authors allow a domain $\Omega\subset\R^2$ with boundary
$\partial\Omega$ that is piecewise of class $C^2$.  It will moreover turn out that the analysis presented
below for the full model~\eqref{e-u}--\eqref{e-IVs} immediately transfers to the more simple
model~\eqref{simplified}. Therefore, the results stated in the present paper are completely new and much
more general than those known so far.
%
\subsection{Outline and strategy}
Our analysis of the system~\eqref{e-u}--\eqref{e-w} fundamentally bases on the fact that it is only
\emph{one} equation,~\eqref{e-u}, where the second derivative of another quantity appears. So we solve
the equations~\eqref{e-v}--\eqref{e-w} for $(v,p,w)$, where $u$ enters parametrically as a given
function. It turns out that the dependence of $(v,p,w)$ on $u$ in this spirit is well-behaved in a
suitable sense. This allows to insert $(v,p,w)$ in their dependence of $u$ into~\eqref{e-u}. Thus, one
ends up with \emph{one} ``scalar'' quasilinear parabolic equation whose dependence on $u$ is
\emph{nonlocal in time}, since the functions $v,p,w$, as solutions to evolution equations themselves,
depend on the \emph{whole} function $u$ on $[0,t]$ instead of just the value $u(t)$. Such an equation,
however, can be solved by a pioneering theorem of Amann which covers such general settings,
cf.~\cite[Thm.~2.1]{AmannDiffEq} or Theorem~\ref{t-Amann} below.  Still, it is a formidable task to
verify the assumptions of the theorem, since the equation under consideration is still quasilinear and
nonlocal in nature.

Thereby it is not obvious a priori in which function spaces the problem should be considered, but since
\emph{homogeneous} Neumann conditions are prescribed, cf.~\eqref{e-neumann}, Lebesgue spaces
$L^p(\Omega)$ are a favorable choice since the boundary conditions are reflected in a strong sense by the
differential operators there, see Remark~\ref{rem:A2} below. Fortunately, there are various recent
elliptic~(\cite{auscher,disser,Ziegler}) and parabolic~(\cite{GKR,HiebRehb}) regularity results available
which are even valid in the case of non-smooth domains and which allow for a treatment
of~\eqref{e-u}--\eqref{e-w} in this setting. The indeed crucial problem is the adequate choice of the
integrability order $p$. However, there is fairly general class of domains $\Omega$ for which the
divergence-gradient operator $-\nabla\cdot\mu\nabla$ admits maximal Sobolev regularity on $W^{1,q}$ for
some $q >d$, that is,
\begin{equation}
  -\nabla \cdot \mu \nabla +1 \colon \quad W^{1,q}(\O) \to
  \bigl(W^{1,q'}(\O)\bigr)'=:W^{-1,q}_\bullet(\Omega)\label{eq:iso_assu} 
\end{equation}
is a topological isomorphism, where $\mu$ is a bounded, measurable and strictly positive function on
$\Omega$, cf.~\cite{disser,Ziegler} (see Chapter~\ref{s-pre} for precise definitions). Combining this
isomorphism property with recent and powerful results on the square root of elliptic operators as
in~\cite[Thm.~5.1]{auscher} (see also Proposition~\ref{p-wurz} below) provides very precise embedding
results for the domains of fractional powers of the elliptic operators on Lebesgue spaces
$L^p(\Omega)$. On the other hand, one can show that the domains of the operators
$-\nabla\cdot \phi \mu \nabla $, when considered on $L^{\frac{q}2}(\O)$, are independent of $\phi$, whenever
$\phi$ is a strictly positive function from $W^{1,q}(\O)$, cf.\ e.g.~\cite{archive} (see also
Lemma~\ref{l-domAIN} below). This is a crucial property in the task of establishing \emph{constant}
domains for the operators entering in the quasilinear equation~\eqref{e-u}, the latter being a central
point in the theorem of Amann mentioned above, for which we then indeed choose a Lebesgue space $L^p(\O)$
with $p = \frac{q}{2}$ for $q>d$ satisfying~\eqref{eq:iso_assu}.

Note that for the Keller-Segel model~\eqref{e-u}--\eqref{e-IVs} one in fact only needs to consider
$\mu \equiv 1$, but our technique is not necessarily restricted to the Laplacian or even only scalar
multipliers within the divergence-gradient operator, cf.\ our comments in Chapter~\ref{s-Concluding} at
the end of the paper.

Let us emphasize that this strategy for the analysis of the system~\eqref{e-u}--\eqref{e-IVs} may be
adopted to both the simplified model~\eqref{simplified} and the situation where the
equations~\eqref{e-v}--\eqref{e-w} for $v,p,$ and $w$ are \emph{elliptic} only, with virtually no
changes. For the latter case, one would even have an immediate relation between $(v(t),p(t),w(t))$ and
$u(t)$ for each $t \in J$, i.e., a local dependency of $(v,t,p)$ on $u$, for which the resulting reduced
equation for $u$ is then tractable using the slightly less restrictive theorem of Pr\"uss~\cite{pruess}
instead of the result of Amann suitable for nonlocal
dependencies. 
See~\cite{ThermistorPre1} for a display of this technique where the (single) elliptic equation is even
also quasilinear.

The outline of the paper is as follows: in the next chapter we will establish notations, general
assumptions and definitions. In Chapter~\ref{s-pre}, we collect preliminary results, partly already
established in other papers. In particular, the concept of maximal parabolic regularity is introduced --
being fundamental for all what follows. The investigation of the model is carried out in
Chapter~\ref{s-inv}, beginning with a precise formulation in Chapter~\ref{s-formu}. The main result,
local-in-time existence and uniqueness for the Keller-Segel system, is formulated in
Theorem~\ref{t-mainres}. It follows the proof of this in Chapter~\ref{s-proof}.  The paper finishes with
concluding comments and remarks in Chapter~\ref{s-Concluding}.

\section[Notations, general assumptions]{Notations, general assumptions and definitions}
The underlying spatial set $\Omega$ is always supposed to be a bounded Lipschitz domain in $\R^d$ for
$d= 2$ or $d=3$ in the sense of~\cite[Def.~1.2.1.2]{grisvard85} or~\cite[Ch.~1.1.9]{mazyasob}. The reader
should carefully notice that this is different from a \emph{strong Lipschitz domain}, which is more
restrictive and in fact identical with a \emph{uniform cone domain}, see
again~\cite[Def.~1.2.1.1]{grisvard85} or~\cite[Ch.~1.1.9]{mazyasob}. We note that a Lipschitz domain has
the extension property, see e.g.~\cite[Thm.~7.25]{gil}, such that the usual function space embeddings are
available.

Concerning function space terminology, $W^{1,q}(\Omega)$ for $q \in \iv{1,\infty}$ stands for the usual
Sobolev space on $\Omega$ as a complex vector space (we will switch to real ones later). Accordingly,
$W_\bullet^{-1,q}(\Omega)$ denotes the \emph{anti-}dual of $W^{1,q'}(\Omega)$.  Moreover, for
$\theta \in \iv{0,1}$ and $q \in \iv{1,\infty}$, $H^{\theta,q}(\O)$ is the symbol for the space of Bessel
potentials on $\O$, cf.~\cite[Ch.~4.2.1]{triebel}. The space of uniformly continuous functions on $\O$ is
denoted by $C(\overline\Omega)$. For an open set $\Lambda \subset \R^N$, where $N \in \{1,2,3\}$, and a
Banach space $X$, we write $C^\alpha(\Lambda;X)$ for the usual $X$-valued H\"older spaces of order
$\alpha\in \iv{0,1}$, cf.~\cite[Ch.~II.1.1.]{A95}. We will mostly encounter these in the incarnations
$\Lambda = \Omega$ and $X = \R$ or $\Lambda$ an interval in $\R$ and $X$ a function space.  Since we
frequently work with triplets of functions, let $\vecL^p(\O)$ and $\mathbb W^{1,q}(\O)$ denote the spaces
$(L^p(\O))^3$ and $(W^{1,q}(\O))^3$, respectively. The domain $\Omega$ under consideration will not
change throughout this work, hence we usually omit the reference to $\O$ when working with the function
spaces.

For two Banach spaces $X$ and $Y$ we denote the space of linear, bounded operators from $X$ into $Y$ by
$\LL(X;Y)$ with $\LL(X) := \LL(X;X)$. The norm in a Banach space $X$ will be always indicated by
$\|\cdot \|_X$. If a Banach space $Y$ is contained in another Banach space $X$ and the canonical injection
of $Y$ into $X$ is continuous, then we say that $Y$ is \emph{embedded} into $X$ and write $Y \embeds
X$. Let $Y$ embed into $X$. Then $\mathcal E(Y;X)$ denotes the \emph{embedding constant}, i.e., the norm
of the embedding map. Moreover, in the same situation, if $B$ is the restriction of an operator
$A \colon X \supseteq \dom(A) \to X$ to the space $Y$, then $\dom_Y(B)$ indicates the domain of this
operator $B$ in $Y$.

Finally, we use $J = \iv{0,T}$ for $0 < T < \infty$, and the letter $c$ denotes a generic constant, not
always of the same value.

\subsection{Assumptions}

In order to allow for concise notation in the later stages of this work, we generalize the nonlinear
growth, production and degradation terms on the right hand sides of~\eqref{e-v}--\eqref{e-w} to general
functions $R_2, R_3, R_4$, including a function $R_1$ for~\eqref{e-u} which is not present in the above
model but poses no problem to include analytically. Note that the differential operator for $v$
in~\eqref{e-u} will be treated specially. For the $R_i$ and for the coefficient functions $\afo$ and
$\bfo$, we make the following assumptions.
\begin{assu} \label{a-Koeff}
  \begin{enumerate}[i)]
  \item The functions $\afo, \bfo\colon\R^2 \to \R$ are supposed to be twice continuously differentiable.
    Moreover, $\afo$ takes only positive values.
  \item For $i = 1,\ldots,4$, each function $R_i$ is defined on $\R^4$ and maps into $\R$, and is also
    assumed to be twice continuously differentiable.
  \end{enumerate}
\end{assu}

We point out that we have to pose another assumption of completely different nature than the above ones
concerning the regularity of the domain $\O$, cf.~Assumption~\ref{a-reg} below. This assumption is only
posed below to put it in the appropriate context.

\begin{rem}
  In the sequel, the functions $\afo, \bfo$ are always readily identified with the induced superposition
  operators, acting from $C(\overline \Omega) \times C(\overline \Omega)$ into $C(\overline \Omega)$. The
  same is, \emph{mutatis mutandis}, done for the functions $R_1, R_2, R_3, R_4$.
\end{rem}
\section[Preliminaries]{Preliminaries: Some operator theoretic results}
\label{s-pre}

In this chapter we declare suitable Banach spaces on which the Keller-Segel system will be considered and
in which the analysis is carried out, and the corresponding differential operators. As already explained
in the introduction, we plan to treat the system in the $L^p$ scale.
Unfortunately, in view of the nonlinearities in the system, the Hilbert space $L^2$ is \emph{not}
appropriate in general, cf.\ also our comments in Chapter~\ref{s-Concluding} below. It will become clear
that $L^p$-spaces with suitably chosen $p$, possibly smaller than $2$, allow for a suitable treatment of
the Keller-Segel system. Thus, it is the aim of the following considerations to provide a consistent
definition of the second order divergence operators on such $L^p$ spaces and to show that these operators
indeed possess suitable functional analytic properties, in particular, maximal parabolic regularity.

\begin{definition} \label{d-coeff} Assume that $\mu$ is a real-valued, measurable, bounded function on
  $\Omega$.  We define
  the continuous linear operator
  \begin{equation*}
    -\nabla \cdot \mu
    \nabla \colon W^{1,2} \to W_\bullet^{-1,2}
  \end{equation*}
  by
  \begin{equation} \label{e-0815} \bigl\langle -\nabla \cdot \mu \nabla v,w\bigr\rangle :=\int_\Omega \mu \nabla
    v\cdot \nabla \overline w \dd \x \quad \text{for } \quad v,w \in W^{1,2}.
  \end{equation}
  It is convenient to view this operator equivalently as a closed one on $W^{-1,2}_\bullet$ with domain
  $W^{1,2 }$. For $q > 2$, we define the operator in $W^{-1,q}_\bullet$ by taking the maximal
  corestriction to that space, thus obtaining again a closed operator, denoted by the same symbols, with
  a generally unknown domain of definition $\dom_{W^{-1,q}_\bullet}(-\nabla\cdot\mu\nabla)$.
\end{definition}

Taking $\mu \equiv 1$ in Definition~\ref{d-coeff}, one, of course, recovers the (negative) weak
Laplacian.

\begin{rem} \label{r-rechtfertig} In this context, it is not quite common to admit functions $\mu$ which
  take positive \emph{and} negative values. Nevertheless, this is unavoidable by the properties of the
  function $\sigma$ originating from the model, cf.~the introduction, see also~\cite{Gajewski:1998}.
\end{rem}

\subsection{The restriction of $-\nabla\cdot\mu\nabla$ to $L^p$ spaces}
\label{sec:restriction--nabla}

Let us in this section consider $-\nabla\cdot\mu\nabla$ as in Definition~\ref{d-coeff} as an operator mapping $W^{1,2}$ to
$W^{-1,2}_{\bullet}$. 
For $p \in [1,\infty[$, we define the \emph{restriction} $A_p(\mu)$ of $-\nabla \cdot \mu \nabla$ to
the space $L^p$ as follows: $\psi \in W^{1,2} \cap L^p$ belongs to $\dom_{L^{p}}(A_p(\mu))$ iff the (anti-) linear
form
\begin{equation} \label{e-antilii} \bigl(W^{1,2} \cap L^{p'}\bigr) \ni \varphi \mapsto \int_\Omega \mu \nabla \psi \cdot \nabla
  \overline \varphi \dd\x = \bigl\langle -\nabla \cdot \mu \nabla \psi ,\varphi \bigr\rangle
\end{equation}
is continuous if $W^{1,2} \cap L^{p'}$ is only equipped with the weaker $L^{p'}$ topology, i.e., if there exists a
constant $c = c(\psi)$ such that
\begin{equation*}
  \bigl|\langle -\nabla\cdot\mu\nabla\psi,\varphi\rangle\bigr| \leq c(\psi) \|\varphi\|_{L^{p'}}
  \quad \text{for all } \varphi \in W^{1,2} \cap L^{p'}.
\end{equation*}
In this case, the functional~\eqref{e-antilii} may be extended by continuity from the dense subspace
$W^{1,2} \cap L^{p'}$ to whole $L^{p'}$ under preservation of its norm. We denote the representative of this
functional on $L^{p'}$ by $\Psi \in L^p$ and define $A_p(\mu) \psi := \Psi$.  Then $A_p(\mu)\psi$
satisfies
\begin{equation} \label{e-constitut} \int_\O \bigl (A_p(\mu)\psi\bigr ) \, \overline \varphi \dd\x
  =\int_\Omega \mu \nabla \psi \cdot \nabla \overline \varphi \dd\x = \bigl\langle -\nabla \cdot \mu \nabla
  \psi ,\varphi \bigr\rangle \quad \text{for all }\varphi \in W^{1,2} \cap L^{p'},
\end{equation}
which is considered as the \emph{constitutive relation} between $-\nabla \cdot \mu \nabla \psi$ and
$A_p(\mu)\psi$. In fact,~\eqref{e-constitut} precisely means that
$-\nabla \cdot \mu \nabla \psi \in W^{-1,2}_\bullet$ is the image of $A_p(\mu)\psi \in L^p$ under the
embedding $L^p \embeds W^{-1,2}_\bullet$.  Moreover, it is clear that the $L^p$-norm of $A_p(\mu)\psi$ is
nothing else but the norm of the antilinear form~\eqref{e-antilii} where $W^{1,2} \cap L^{p'}$ is equipped with the
$L^{p'}$-norm.

Since the notation $A_p(\mu)$ already indicates the space on which the operator is assumed to act, we
write $\dom(A_p(\mu))$ instead of $\dom_{L^p}(A_p(\mu))$ if there is no need for greater care. Note that
the often used technique to construct the ``strong'' differential operators on the $L^p$ scale by
restricting $A_2(\mu)$ to $L^p$ for $p > 2$ and taking adjoints of these resulting operators to define
the corresponding operator in $L^p$ for $p < 2$ (or forming the closure of $A_2(\mu)$ there) gives the
same operators as the procedure above.

We will mostly consider the case of strictly positive $\mu$; only in Lemma~\ref{l-domAIN} properties of
the operators $A_p(\mu)$ with possibly nonpositive values for $\mu$ are pointed out which are
fundamental for the treatment of the divergence operator in the right hand side of~\eqref{e-u}. Hence,
let us now assume for the rest of this subchapter that $\mu$ is bounded from below by a positive
constant.

\begin{rem}\label{rem:A2}
  It is well-known that the property $\psi \in \dom(A_2(\mu))$ implies a (generalized) homogeneous
  Neumann condition $\nu \cdot \mu \nabla \psi =\nu \cdot \nabla \psi =0$ on $\partial \Omega $,
  cf.~\cite[Ch.~1.2]{cia} or~\cite[Ch.~II.2]{ggz}, $\nu$ being the outer normal at the boundary.  This
  fact reflects the homogeneous Neumann boundary conditions~\eqref{e-neumann} on the functional analytic
  level.
\end{rem}

We collect some properties of the operators $A_p(\mu)$ and its relation with $-\nabla\cdot\mu\nabla$.

\begin{proposition} Let $\mu \in L^\infty(\O)$ be a real function with a strictly positive lower
  bound. Then the Lipschitz property of $\Omega$ implies the following assertions:
  \label{prop:collection-Ap}
  \begin{enumerate}[i)]
  \item The operator $A_2(\mu)$ is a non-negative, selfadjoint operator on $L^2$, classically considered
    as the operator induced by the form~\eqref{e-0815} on $W^{1,2}$.
  \item Under the Lipschitz assumption on $\Omega$, the operators $\nabla\cdot\mu\nabla$ generate
    analytic semigroups on $W^{-1,q}_\bullet$ for all $q \in [2,\infty[$.
  \item $-A_2(\mu)$ generates a contractive semigroup $\{\sg{-tA_2(\mu)}\}_{t \ge 0}$ on $L^2$ which
    extends consistently to all $L^p$ spaces for $p \in [1,\infty]$ and is moreover analytic if
    $p < \infty$. These semigroups are also consistent with the ones generated by $\nabla\cdot\mu\nabla$ on
    $W^{-1,q}_\bullet$ and their generators are exactly the operators $-A_p(\mu)$. The semigroups $\sg{-t(A_p(\mu)+1)}_{t\geq0}$ transform \emph{real} functions into \emph{real}
    ones and \emph{positive} ones into \emph{positive} ones.
  \item Both $-\nabla\cdot\mu\nabla+1$ on $W^{-1,q}_\bullet$ for $q \in [2,\infty[$ and $A_p(\mu)+1$ on
    $L^p$ for $p \in \iv{1,\infty}$ are positive operators; in particular, their fractional powers are
    well-defined. The operators $-A_p(\mu) + 1$ even admit bounded imaginary powers: the
    set of operators
    $\bigl\{(A_p(\mu)+1)^{\mathrm{i}s} \colon s \in \iv{-\varepsilon, \varepsilon}\bigr\}$ is bounded in
    $\LL(L^p)$ for every $p \in \iv{1,\infty}$ and every $\varepsilon >0$.
  \item The operator $A_2(\mu)+1$ satisfies the Kato square root property, that is, we have
    $\dom\bigl((A_2(\mu)+1)^{\frac12}\bigr) = W^{1,2}$, or equivalently, $(A_2(\mu)+1)^{\frac12}$ is a
    topological isomorphism between $W^{1,2}$ and $L^2$.
  \end{enumerate}
\end{proposition}


\begin{proof}
  \begin{enumerate}[i)]
  \item See~\cite[Ch.~1.2.3]{Ouh05} or the classical text~\cite[Ch.~VI.2]{kato}.
  \item See~\cite[Lem.~6.9(c)]{TomKaroJoDiffEq}.
  \item The extension of $\{\sg{-tA_2(\mu)}\}_{t \ge 0}$ to $L^p$ is proven in~\cite[Corollaries~2.16
    and~4.10]{Ouh05}. Consistency of the $L^p$ semigroups is shown in~\cite[Ch.~1.4.2]{Ouh05}, whereas
    consistency with the $W^{-1,q}_\bullet$-scale can be found in~\cite[Ch.~4]{tomjo-consist}. That
    $-A_p(\mu)$ is the generator of the $L^p$ semigroups follows from the constitutive
    relation~\eqref{e-constitut} and~\cite[Prop.~2.5]{tomjo-consist}. The mapping properties for real and
    positive functions are from~\cite[Ch.~2.6]{Ouh05}.
  \item The positive operator property for the $W^{-1,q}_\bullet$ operators follows from the same
    property for the $L^p$ operators, cf.~\cite[Thm.~11.5]{auscher}, which then implies well-definedness
    of their fractional powers by~\cite[Ch.~1.15]{triebel}. For the bounded imaginary powers,
    see~\cite{cowl} or~\cite[Cor.~7.24]{Ouh05}.
  \item This is the classical result of Kato~\cite[Ch.~5]{katosr} in conjunction with $A_2(\mu)$ being
    selfadjoint. \qedhere
  \end{enumerate}
\end{proof}


\begin{rem}
  \label{rem:domAp}
  The domain of the operator $A_p(\mu)$ is always equipped with the usual norm
  $\|(A_p(\mu) +1)\cdot\|_{L^p}$, or $\|(A_p(\mu) +1)\cdot \|_{\vecL^p}$ when considered on the space
  $L^p$ or $\vecL^p$, respectively. This means that $\dom A_p(\mu)$ and $\dom\bigl(A_p(\mu) + 1\bigr)$
  coincide as Banach spaces and we will use them interchangeably.
\end{rem}

Observing that the fractional powers of $-\nabla\cdot\mu\nabla +1$ and $A_p(\mu)+1$ are well-defined, the
boundedness of the imaginary powers of $A_p(\mu) + 1$ in particular implies the identity of the domains
of fractional powers $(A_p(\mu) + 1)^\alpha$ with interpolation spaces between $L^p$ and
$\dom(A_p(\mu) + 1)$, see~\cite[Ch.~1.15.3]{triebel} or~\cite[Ch.~4.6/4.7]{A95}. We devote a subchapter
to the special fractional powers which we need in the following.

\subsection{Fractional powers of the elliptic operators}
\label{sec:fracpower}

In this section, we ultimately establish the embedding
\begin{equation}
  \dom \bigl ((A_p(\mu)+1)^{\frac {1}{2} + \frac {d}{2q}} \bigr )\embeds W^{1,q}\label{eq:crucEmbedding}
\end{equation}
for some $q > d$ with $p \geq \frac{q}{2}$,
cf.~Theorem~\ref{t-multiplier} below. The main tool here, which will be the ``anchor'' in the derivation
of~\eqref{eq:crucEmbedding}, is the precise information on the domain of definition of the square root of
the operators $-\nabla\cdot\mu\nabla+1$, cf.\ Proposition~\ref{p-wurz}, together with the following
assumption, which essentially allows to ``lift'' the obtained regularity to sufficiently high levels:
\begin{assu} \label{a-reg} There is a $q > d$ such that
  \begin{equation} \label{e-topisoq} -\nabla\cdot\nabla+1\colon W^{1,q} \to W_\bullet ^{-1,q}
  \end{equation}
  provides a topological isomorphism, the operator being defined as in
  Definition~\ref{d-coeff}. Equivalently,~\eqref{e-topisoq} being a continuous isomorphism means that
  $\dom_{W^{-1,q}_\bullet}(-\nabla\cdot\nabla + 1)$ is exactly $W^{1,q}$.
\end{assu}

We suppose Assumption~\ref{a-reg} to be satisfied for the rest of this work and fix the corresponding
number $q \in \iv{d,4}$.

Since Assumption~\ref{a-reg} in fact implicitly determines the class of admissible domains, an
(extensive) comment on this should be in order:

\begin{rem} \label{r-commentdomain}
  \begin{enumerate}[i)]
  \item In case of $d=2$, the assumption is fulfilled for any Lipschitz domain $\Omega$. This is the main
    result in the classical paper~\cite{groeger89}, there even established for mixed boundary conditions.
  \item It is exactly this condition which---besides the \emph{a priori} required Lipschitz
    property---puts a restriction on the geometry of the underlying domain $\Omega$ in three spatial
    dimensions in this paper. For $d = 3$, it is known that Assumption~\ref{a-reg} holds true in case of
    \emph{strong} Lipschitz domains $\Omega$, cf.~\cite{Zanger}. Moreover, it is also true for Lipschitz
    domains $\Omega$ whose closures form---generally nonconvex---polyhedrons, cf.~\cite{Ziegler}. Note
    that this latter class is, by far, \emph{not} contained in the class of strong Lipschitz domains, as
    the (topologically regularized) double beam shows.
  \item Assumption~\ref{a-reg} is also fulfilled for domains which are obtained locally as $C^1$
    deformations of the ones mentioned before.
  \item It is well-known that, even for strong Lipschitz domains, the admissible index $q$ exceeds $3$ by
    an arbitrarily small margin only, cf.~\cite[Introduction]{Zanger}, cf.\ also~\cite[Thm.~A]{jer/ke}.
    In case of $C^1$-domains $\Omega$, $q$ may be chosen arbitrarily large (cf.~\cite[Section~15]{ADN}
    or~\cite[p.~156--157]{morr}); but if one admits polyhedral domains the isomorphism index $q$ cannot
    be expected to be larger than $4$ in general, since edge and corner singularities
    appear,~cf.~\cite{Dauge1},~\cite{Dauge2}. See also~\cite{mercier} and~\cite[Appendix]{HKR} for sharp
    estimates of edge singularities.
  \item If $\phi$ is a uniformly continuous function on $\Omega$ with a positive lower bound, then
    Assumption~\ref{a-reg} implies that
    \begin{equation} -\nabla \cdot \phi \nabla+1\colon W^{1,q} \to W^{-1,q}_\bullet\label{e-isomro}
    \end{equation}
    is also a topological isomorphism, cf.~\cite[Ch.~6]{disser}.
  \end{enumerate}
\end{rem}

Altogether, this shows that Assumption~\ref{a-reg} is fulfilled for a fairly rich class of domains which
should cover almost all interesting constellations in the applications.


The following recent result on the regularity properties of the square root of
$-\nabla \cdot \mu \nabla +1$ is, in cooperation with the isomorphism~\eqref{e-topisoq}, the central
instrument for deriving estimates for suitable fractional powers of the differential operators.

\begin{proposition} \label{p-wurz} Let $\mu$ denote any real, measurable function on $\Omega$ which is
  bounded from below and above by positive constants.
  \begin{enumerate}[i)]
  \item The isomorphism $(A_2(\mu) + 1)^{-\frac {1}{2}}\colon L^2 \to W^{1,2}$, cf.\
    Proposition~\ref{prop:collection-Ap}, continuously extends to an isomorphism from $L^p$ onto
    $W^{1,p}$ for $p \in \iv{1,2}$. Hence, 
    the operator
    ${(A_p(\mu)+1)}^{\frac{1}{2}}$ provides a topological isomorphism between the spaces $W^{1,p}$ and
    $L^p$, or, in other words: $\dom( {A_p(\mu)+1})^{\frac{1}{2}}=W^{1,p}$, for all
    $p \in \iv{1,2}$.
  \item $({-\nabla \cdot \mu \nabla +1})^{\frac{1}{2}}$ provides a topological isomorphism between the
    spaces $L^p$ and $W^{-1,p}_\bullet$, in other words:
    $\dom_{W^{-1,p}_\bullet}( {-\nabla \cdot \mu \nabla +1})^{\frac{1}{2}}= L^p$, for all
    $p \in [2,\infty[$.
  \item We have
    \begin{equation}
      \label{eq:domAtheta}
      \dom \bigl ((A_p(\mu) +1)^{\frac{\theta}{2}}\bigr )= H^{\theta,p}
    \end{equation}
    for
    $p \in \mathopen]1,2\mathclose]$ and $\theta \in \iv{0,1} \setminus \{\frac {1}{p}\}$.
  \end{enumerate}
\end{proposition}
\begin{proof}
  i)~is the main result in~\cite{auscher}, cf.\ Thm.~5.1 there. ii)~follows from~i) by duality because
  $A_2(\mu)$ is selfadjoint on $L^2$, see Proposition~\ref{prop:collection-Ap}. iii)~ Since $A_p(\mu)+1$
  admits bounded imaginary powers (again, Proposition~\ref{prop:collection-Ap}),
  \begin{equation*}
    \dom \bigl ((A_p(\mu) +1)^{\frac{\theta}{2}}\bigr )=
    \bigl[L^p,\dom( {A_p(\mu) +1})^{\frac{1}{2}}\bigr]_\theta
  \end{equation*}
  follows from~\cite[Ch.~1.15.3]{triebel}. By~i), the latter is equal to $[L^p,W^{1,p}]_\theta$, and this
  space is exactly $H^{\theta,p}$ as proved in~\cite[Thm.~3.1]{ggkr}.
\end{proof}

\begin{lemma} \label{l-wurzcoroll} Let $\phi$ denote any real, uniformly continuous function on $\Omega$
  which is bounded from below by a positive constant.  Then, under Assumption~\ref{a-reg},
  $\bigl({A_p(\phi) +1}\bigr)^{\frac{1}{2}}$ provides a topological isomorphism between $W^{1,p}$ and
  $L^p$ for all $p \in \mathopen]2,q]$, and~\ref{eq:domAtheta} for $\mu = \phi$ holds true for this range
  of $p$ as well.
\end{lemma}

\begin{proof}
  First of all, Remark~\ref{r-commentdomain} tells us that under the given supposition on $\phi$,
  Assumption~\ref{a-reg} implies the isomorphism property~\eqref{e-isomro}, which then also holds true
  for all $p \in [2,q]$ due to interpolation. Having this at hand, the isomorphism property for the
  square root operators follows in a straight forward manner from Proposition~\ref{p-wurz}~ii) for
  $\mu = \phi$, see also~\cite[Thm.~6.5]{TomKaroJoDiffEq}. This also implies~\eqref{eq:domAtheta} for
  $p \in \mathopen]2,q]$ with the same proof as in Proposition~\ref{p-wurz}.
\end{proof}

The square root isomorphisms and identity~\eqref{eq:domAtheta} from Lemma~\ref{l-wurzcoroll} have the
following immediate consequence:

\begin{theorem} \label{t-multiplier} Let $\phi$ denote any real, uniformly continuous function on
  $\Omega$ which is bounded from below by a positive constant.  Then, for every $p \ge \frac {q}{2}$ one
  has the embedding
  \begin{equation} \label{eq:embedddd} \dom \bigl ((A_p(\phi)+1)^{\frac {1}{2} + \frac {d}{2q}} \bigr
    )\embeds W^{1,q}, \end{equation} which implies
  \begin{equation}
    \label{e-embeddd}
    \bigl (L^p,\dom(A_p(\phi))\bigr)
    _{\theta,1}\embeds W^{1,q},
  \end{equation}
  for all $\theta \in \bigl[\frac12+\frac {d}{2q},1\bigr[$.
\end{theorem}

\begin{proof}
  The bounded imaginary powers of $A_p(\phi) + 1$, cf.\ Proposition~\ref{prop:collection-Ap}, imply that
  \begin{equation*}
   \bigl (L^p,\dom(A_p(\phi)+1)\bigr)_{\theta,1} 
    \embeds  \bigl [L^p,\dom(A_p(\phi)+1)\bigr]_{\frac12+\frac {d}{2q}} 
  =    \dom\bigl((A_p(\phi)+1)^{\frac12+\frac {d}{2q}}\bigr)
  \end{equation*}
  for all $\theta \in \bigl[\frac12+\frac {d}{2q},1\bigr[$, see~\cite[Ch.~1.15.3]{triebel}. In this
  sense,~\eqref{e-embeddd} is a direct consequence of~\eqref{eq:embedddd}, modulo identification of
  $\dom A_p(\phi)$ and $\dom\bigl(A_p(\phi)+1\bigr)$.
  We show that~\eqref{eq:embedddd} holds true by
  proving that $(A_p(\phi)+1)^{-(\frac {1}{2} + \frac {d}{2q})}$ is a continuous linear operator from $L^p$ to $W^{1,q}$ for
  these $\theta$. We split the operator as follows:
  \begin{multline}
    \bigl\|(A_p(\phi) +1)^{-(\frac {1}{2} + \frac {d}{2q})}\bigr\|_{\LL(L^p;W^{1,q})} \\ \le \bigl\|(A_q(\phi) +1)^{-\frac
      {1}{2}}\bigr\|_{\LL(L^q;W^{1,q})}
    \bigl\|(A_p(\phi)+1)^{-\frac {d}{2q}}\bigr\|_{\LL(L^p;L^q)}.
    \label{e-inversepow}
  \end{multline}
  Thanks to Lemma~\ref{l-wurzcoroll}, it remains to show that $(A_p(\phi)+1)^{-\frac{d}{2q}}$ is a
  continuous linear operator from $L^p$ to $L^q$. We show that
  $\dom\bigl((A_p(\phi)+1)^{\frac{d}{2q}}\bigr) \embeds L^q$. For $p > q$, we always have
  (cf.~\cite[Thm.~1.15.2]{triebel})
    \begin{equation*}
       \dom\bigl((A_p(\phi)+1)^{\theta}\bigr)  \embeds
      \bigl(L^p,\dom(A_p(\phi)+1)\bigr)_{\theta,\infty} \embeds L^p \embeds L^q.
    \end{equation*}
    For $p \in \bigl[\frac{q}2,q\bigr]$ in turn, Proposition~\ref{p-wurz} and Lemma~\ref{l-wurzcoroll}
    yield $\dom((A_p(\phi) + 1)^{\frac{d}{2q}}) = H^{\frac{d}{q},p}$ which exactly embeds into
    $L^q$. Hence, $(A_p(\phi)+1)^{-\frac{d}{2q}} \in \LL(L^p;L^q)$ in both cases, and
    from~\eqref{e-inversepow} we obtain that
  \[
    \dom \bigl ((A_p(\phi)+1)^{\frac12+\frac {d}{2q}} \bigr ) \embeds W^{1,q},
  \]
  which was the claim.
\end{proof}

\subsection{Maximal parabolic regularity and consequences for nonlinear problems}
We next introduce preparatory concepts and results concerning parabolic operators.  Throughout the rest
of this paper let $T > 0$ and set $J = \iv{0,T}$.  First, we introduce the Bochner-Sobolev spaces.

\begin{definition}
  If $X$ is a Banach space and $r \in \iv{1,\infty}$, then we denote by $L^r(J;X)$ the space of
  $X$-valued functions $f$ on $J$ which are Bochner-measurable and for which $\int_J\|f(t)\|_X^r\dd t$ is
  finite.  We define the Bochner-Sobolev spaces
  \begin{equation*}
    W^{1,r}(J;X):=\bigl\{u \in L^r(J;X)\colon u' \in L^r(J;X)\bigr\},
  \end{equation*}
  where $u'$ is to be understood as the time derivative of $u$ in the sense of $X$-valued distributions
  (cf.~\cite[Section~III.1]{A95}).  Moreover, we introduce the subspace of functions with initial value
  zero $W_0^{1,r}(J;X):= \{ \psi \in W^{1,r}(J;X) \colon \psi(0)=0\}$.
\end{definition}

Let us define a suitable notion of maximal parabolic regularity in the non-autonomous case and point out
some basic facts on this:
\begin{definition} 
  Let $X$, $D$ be Banach spaces with $D$ densely embedded in $X$.  Let
  $J \ni t \mapsto \mathcal A(t) \in \LL(D;X)$ be a bounded and measurable map and suppose that
  the operator $\mathcal A(t)$ is closed in $X$ for all $t \in J$.  Let $r \in \iv{1,\infty}$.  Then we
  say that the family $\{\mathcal A(t)\}_{t \in J}$ satisfies \emph{(non-autonomous) maximal parabolic
    $L^r(J;D,X)$-regularity}, if for any $f \in L^r(J;X)$ there is a unique function
  $u \in L^r(J;D) \cap W_0^{1,r}(J;X)$ which satisfies
  \begin{equation} \label{e-0paragleich} u'(t) +\mathcal A(t)u(t)=f(t)
  \end{equation}
  for almost all $t \in J$.  We write
  \[
    \mr^r(J;D,X) := L^r(J;D) \cap W^{1,r}(J;X)\] and \[ \mr^r_0(J;D,X) := L^r(J;D) \cap W_0^{1,r}(J;X)
  \]
  for the spaces of maximal parabolic regularity.
  From the open mapping theorem, we further obtain that there exists a constant $c$ such that
  \begin{equation}
    \label{eq:MPRestimate}
    \|u\|_{\mr^r_0(J;D,X)} \leq c \|f\|_{L^r(J;X)}
  \end{equation}
  for all $f \in L^r(J;X)$ and $u$ being the associated unique solution of~\eqref{e-0paragleich}.
\end{definition}

If all operators $\mathcal A(t)$ are equal to one (fixed) operator $\mathcal A_0$, and there exists an
$r \in \iv{1,\infty}$ such that $\{\mathcal A(t)\}_{t \in J}$ satisfies maximal parabolic
$L^r(J;D,X)$-regularity, then $\{\mathcal A(t)\}_{t \in J}$ satisfies maximal parabolic
$L^s(I;D,X)$-regularity for all $s \in \iv{1,\infty}$ and all other (finite) intervals $I$
(cf.~\cite{dore}), and we say that $\mathcal A_0$ satisfies \emph{maximal parabolic regularity on $X$}.

The following embedding result for the spaces of maximal parabolic regularity is essentially used in the
sequel.

\begin{lemma} \label{l-maxparregfacts} Let $X, Y$ be two Banach spaces, with dense embedding
  $Y \embeds X$, and let $r \in \iv{1,\infty}$.
  \begin{enumerate}[i)]
  \item There is an embedding
    \begin{equation} \label{e-embedcont} \mr^r(J;Y,X) \embeds C\bigl(\overline{J}; (X,
      Y)_{1-\frac{1}{r},r}\bigr).
    \end{equation}
  \item Conversely, if the operator $A$ generates an analytic semigroup on the Banach space $X$ with $Y$
    as its domain, and $\psi \in (X, Y)_{1-\frac{1}{s},s}$, then the function $\sg{\cdot A}\psi$ belongs
    to $W^{1,s}(J; X) \cap L^s(J;Y)$ for every bounded interval interval $J=[0,T[$.
  \item There is an embedding
    \begin{equation} \label{e-maxreghoelderembed} \mr^r(J;Y,X) \embeds
      C^\alpha\bigl(J;(X,Y)_{\varrho,1}\bigr)
    \end{equation}
    where $0 < \alpha = 1-\varrho - \frac {1}{r}$.
  \end{enumerate}
\end{lemma}
\begin{proof}
  i)~is proved in~\cite[Ch.~4.10]{A95}, ii)~is shown in~\cite[Ch.~2.2.1 Prop.~2.2.2]{luna}, and iii)~is
  proved in~\cite[Ch.~3,~Thm.~3]{A01}, see also~\cite{TomKaroJo} for a simple proof.
\end{proof}

In the immediate context of maximal parabolic regularity, $Y$ is taken as $\dom_X(A)$ equipped with the
graph norm, of course.

\begin{rem}\label{rem:IVpassend}
  The first two points of Lemma~\ref{l-maxparregfacts} together show that the space
  $(X, \dom_X(A))_{1-\frac{1}{r},r}$, is the adequate space of initial values in the framework of maximal
  parabolic regularity.
\end{rem}

Moreover, we need the following results.
\begin{theorem}[{\cite[Thm.~2.5]{schnaubelt}}] \label{t-pruessschnaub} Let the following two suppositions
  be satisfied:
  \begin{enumerate}[(H1)]
  \item
    The family of operators $\{\mathcal A(t)\}_{t \in \overline J}$, acting on a Banach space $X$ has a
    common dense domain $D$ and the mapping
    $\overline J \ni t \mapsto \mathcal A(t) \in \LL(D; X)$ is continuous. Moreover, each operator
    $\mathcal A(\tau)$, $\tau \in \overline J$, generates an analytic semigroup on $X$.
  \item
    For some $r \in \iv{1,\infty}$, every (fixed) $\tau \in [0,T]$ and all $f \in L^r(J;X)$ there is a
    unique element $u \in \mr^r_0(J;D;X)$ which satisfies the equation $u'+\mathcal A(\tau)u =f$.
  \end{enumerate}
  Then $\{\mathcal A(t)\}_{t \in \overline J}$ satisfies maximal parabolic $L^r(J;D,X)$-regularity.
\end{theorem}

\begin{theorem} \label{t-unsere} Let $\mu$ be a real, bounded, measurable function on $\Omega$ which
  admits a positive lower bound. Then, for every $p \in \iv{1,\infty}$, the operators $A_p(\mu)$ admit
  maximal parabolic regularity on $L^p$. 
\end{theorem}
\begin{proof}
  The theorem can be proved in different ways: in~\cite[Thm.~5.4]{HiebRehb} it is shown via Gaussian
  estimates for the heat kernel, heavily resting on~\cite{HiebPruess}, see
  also~\cite{coulhon}. Alternatively, the theorem is proved in~\cite[Ch.~7]{GKR}, there resting on the
  contractivity of the induced semigroups on all $L^p$ spaces (cf.\ Proposition~\ref{prop:collection-Ap})
  and the pioneering result of Lamberton~\cite{lambert}. The latter allows to prove maximal parabolic
  regularity on even more general Lebesgue spaces, see~\cite{elstmeyrrehb}.
\end{proof}

\begin{theorem}[{\cite[Thm.~2.1]{AmannDiffEq}}] \label{t-Amann} Let 
  $r \in \iv{1,\infty}$ and suppose that $X, Y$ are Banach spaces with dense embedding $Y \embeds
  X$. Also assume the following:
  \begin{enumerate}[i)]
  \item \label{t-Amann-evil}$\mathcal A$ is a map from $\mr^r(J;Y,X)$ into $L^\infty(J;\LL(Y;X))$,
    the latter space being identified with a subset of the \emph{non-autonomous} parabolic operators on
    $X$. Moreover, $\mathcal A$ is Lipschitz continuous on bounded subsets.
  \item For each $u \in \mr^r(J;Y,X)$ and every $S \in \mathopen]0,T]$ the non-autonomous operator
    $\mathcal A(u)|_{]0,S[}$ provides a topological isomorphism between $\mr_0^r(0,S;Y,X)$ and
    $L^r(0,S;X)$.
  \item The mapping $F\colon\mr^r(J;Y,X) \to L^s(J;X)$ is Lipschitzian on every bounded subset for some
    $s > r$.
  \item Both $\mr^r(J;Y,X) \ni u \mapsto \mathcal A(u) \in L^\infty(J;\LL(Y;X))$ and
    $F \colon \mr^r(J;Y,X) \to L^s(J;X)$ are \emph{Volterra maps}, i.e.
    \[u|_{\iv{0,S}}=v|_{\iv{0,S}} \quad \implies \quad \bigl(A(u),F(u)\bigr)|_{\iv{0,S}}=
      \bigl(A(v),F(v)\bigr)|_{\iv{0,S}}\] for every $S \in \iv{0,T}$.
  \item $u_0 \in (X,Y)_{1-\frac {1}{r},r}$.

\end{enumerate}
Then there is a (maximal) interval $I_\bullet:= \iv{0,S_\bullet} \subseteq J$ such that the equation
\[
  u' +\mathcal A(u)u =F(u) , \quad u(0)=u_0
\]
has a solution $u$ on every subinterval $I = \iv{0,S} \subseteq I_\bullet$ which belongs to the maximum
regularity space $\mr^r(I;Y,X)$. Moreover, this solution is unique.
\end{theorem}
\begin{rem} 
  It is known since long that the Volterra property allows to derive results which are not available in a
  more general context without this property, see e.g.~\cite[Ch.~V]{ggz}.  Nevertheless, we feel that
  Amann's result is very close to the ``optimum'' what can be achieved. The reader is advised to consult~\cite[Thm.~3.1]{AmannRuss} for comments on the result by its inventor and a (fixable) shortcoming in
  the proof in~\cite{AmannDiffEq}.
\end{rem}

\subsection{Transferring to real spaces}
\label{sec:real}
Up to now, we have worked in a \emph{complex} setting, but the Keller-Segel system has to be read as a
\emph{real} one. Therefore we transfer the results which we need in the sequel to the corresponding real
spaces. In order to do this, we denote the real parts of $L^p$ and $W^{1,q}$ by $L^p_\R$ and
$W_\R^{1,q}$.

\begin{rem} 
  The necessity to start with complex spaces and to re-evaluate the assertions to also hold in the real
  case can be explained as follows: Most results up to this chapter~\ref{sec:real} are complex in their
  very nature, a particular example being Proposition~\ref{p-wurz}. This makes it evident that, at this
  point, complex spaces are the correct setting. On the other hand, the condition of being twice
  continuously differentiable for the nonlinear functions is more or less inevitable in our context as
  will become clear below, cf.\ Lemma~\ref{l-Lii;ps}, Corollary~\ref{c-lipsdchd} and
  Lemma~\ref{l-Lipschitzconty}. But imposing this condition in a \emph{complex} setting in fact
  necessitates the \emph{analyticity} of the corresponding functions, which is drastically and more
  importantly unnecessarily more restrictive. Hence we ``do the twist'' and switch to real spaces for the
  actual investigation of the model.
\end{rem}

The starting point is the insight that the semigroup operators $\sg{-tA_p(\mu)}$ map real functions into
real functions if the coefficient function $\mu$ is real-valued, as noted in
Proposition~\ref{prop:collection-Ap}. Hence, the operators $(A_p(\mu)+\lambda)^{-1} \colon L^p \to L^p$
\emph{also} map real functions into real ones if $\lambda \in \iv{0,\infty}$.  This makes clear that the
operator $A_p(\mu) $ has a meaningful restriction to $L^p_\R$, whose domain also consists of real
functions only. We denote this domain by $\dom_\R(A_p(\mu))$ for the rest of this subsection.

\begin{lemma} \label{l-carryover} Let $\phi$ be a real, uniformly continuous function which is bounded
  from below by a positive constant.  The assertion of Theorem~\ref{t-multiplier} remains true in case of
  real spaces, i.e., one has for $ p \ge \frac {q}{2}$ the embedding
  \begin{equation} \label{e-realembed}
    \bigl(L_\R^p,\dom_\R(A_p(\phi)\bigr)_{\theta,1}\embeds
    W_\R^{1,q} \embeds C(\overline \Omega)
  \end{equation}
  for all $\theta \in \bigl[\frac12+\frac {d}{2q},1\bigr[$.
\end{lemma}
\begin{proof}
  Let us first recall (see Remark~\ref{rem:domAp}) that we have topologized $\dom_\R(A_p(\phi))$ by the
  norm $\|(A_p(\phi) +1)\cdot \|_{L^p_\R}$. Further, by Theorem~\ref{t-multiplier}, there is a positive
  constant $c$ such that the following inequality holds true for all $\psi \in \dom(A_p(\phi))$ and
  $\theta \in [\frac12 + \frac{d}{2q},1[$:
  \begin{equation} \label{e-interineq} \|\psi\|_{W^{1,q}} \le c \,\|\psi \|_{L^p}^{1-\theta}\, \|\psi
    \|^{\theta} _{\dom(A_p(\phi))} = c \,\|\psi \|_{L^p}^{1-\theta}\, \bigl\|(A_p(\phi) +1)\psi
    \bigr\|^{\theta} _{L^p}.
  \end{equation}
  In particular, inequality~\eqref{e-interineq} is true for every \emph{real} function
  $\psi \in \dom_\R(A_p(\phi))$, and then reads
  \begin{equation} \label{e-interineqreal} \|\psi\|_{W_\R^{1,q}} \le c \,\|\psi \|_{L_\R^p}^{1-\theta}\,
    \bigl\|(A_p(\phi) +1)\psi \bigr\|^{\theta}_{L_\R^p}\,= c \, \|\psi \|_{L_\R^p}^{1-\theta} \|\psi
    \|^{\theta} _{\dom_\R(A_p(\phi))}.
  \end{equation}
  But~\eqref{e-interineqreal} is constitutive for the embedding~\eqref{e-realembed},
  cf.~\cite[Ch.~3.5]{bergh} or~\cite[Ch.~5, Prop.~2.10]{bennet}.
\end{proof}

\begin{theorem} \label{t-unsreal} Let $\mu$ be a real, bounded, measurable function on $\Omega$ which
  admits a positive lower bound. Then, for every $p \in \iv{1,\infty}$, $A_p(\mu)$ admits maximal
  parabolic $L_\R^p$ regularity. 
\end{theorem}
\begin{proof}
  Let $f \in L_\R^p$. Then, by maximal parabolic $L^p$ regularity of $A_p(\mu)$, there exists a unique
  function $u \in \mr_0^r(J;\dom(A_p(\mu)),L^p)$ such that
  \[
    u'(t) +A_p(\mu) u(t) = f(t) \quad \text{in } L^p \quad \text{for almost all } t \in J
    .
  \]
  But then this solution is given by the variation-of-constants formula
  \[
    u(t)=\int_0^t \sgb{-(t-s) A_p(\mu)} f(s) \dd s,
  \]
  and since the semigroup operators transform real functions into real ones, cf.\
  Proposition~\ref{prop:collection-Ap}, it is clear that the solution in fact belongs to the space
  $W_0^{1,r}(J;L_\R^p) \cap L^r(J;\dom_\R(A_p(\mu)))$, what proves the claim.
\end{proof}

Switching to real spaces, the symbol $\dom(A_p(\mu))$ from now on denotes the domain of $A_p(\mu)$
considered on the \emph{real} space $L^p_\R$.

\subsection{Constant domains for $A_p(\varphi)$}

We will need that the domains of the differential operators $A_p(\varphi)$ are uniform w.r.t.\ $\varphi$
from a certain regularity class, as per the assumptions in Theorem~\ref{t-Amann}.  In general, this is
not to be expected if $\varphi$ does not have a positive lower bound, cf.\ also
Remark~\ref{r-rechtfertig}. Still, we need that the differential operator on the right-hand side
in~\eqref{e-u}, which is the one having potentially nonpositive coefficient function values, is
compatible with the domain of definition for the function $v(t)$.

It will turn out that both the latter \emph{and} the constant domain of definition for the differential
operators on the left-hand side in~\eqref{e-u} is exactly $\domLap$. We prove the following lemma which
covers all these considerations in its generality, there writing $\Delta $ instead of $-A_p(1)$ and
already supposing that all occurring spaces are in fact \emph{real} ones.

\begin{lemma} \label{l-domAIN} Let $p=\frac{q}{2}$ and assume $\rho \in W^{1,q}$. Then the following
  assertions hold true:
  \begin{enumerate}[i)]
  \item The domain of the Laplacian is embedded into the domain of $A_p(\rho)$, that is,
    \begin{equation*}
      \domLap \embeds \dom(A_p(\rho)).
    \end{equation*}
  \item If $\rho$ has, additionally, a positive lower bound, then the reverse embedding
    \begin{equation*}
      \dom(A_p(\rho)) \embeds \domLap
    \end{equation*}
    is also true, and $\domLap$ and $\dom(A_p(\rho))$ coincide as Banach spaces.
  \end{enumerate}
\end{lemma}

\begin{proof}
  i)~Let $\psi \in \domLap$ and consider the linear form
  \begin{equation} \label{e-linformlp} \bigl(W^{1,2} \cap L^{p'}\bigr) \ni \varphi \mapsto \langle -\nabla \cdot \rho \nabla
    \psi, \varphi \rangle.
  \end{equation}
  We show that $\psi \in \dom(A_p(\rho))$ by showing that~\eqref{e-linformlp} is continuous w.r.t.\ the
  $L^{p'}$-topology. 
  Therefore we estimate
  \begin{align} \label{e-eessti} \Big |\int_\Omega \rho \nabla \psi \cdot \nabla \varphi\dd\x \Big | & =
    \Big |\int_\Omega \nabla \psi \cdot \nabla (\rho \varphi) \dd\x - \int_\Omega \varphi \nabla \psi
    \cdot \nabla \rho \dd\x \Big | \\ \notag & \leq \Big |\int_\Omega \nabla \psi \cdot \nabla (\rho
    \varphi) \dd\x \Big | + \Big |\int_\Omega \varphi \nabla \psi \cdot \nabla \rho \dd\x \Big | \\
    \notag & \le \|\rho\|_{L^\infty} \|\Delta \psi\|_{L^p}\|\varphi \|_{L^{p'}} + \|\nabla \psi \|_{L^q}
    \|\nabla \rho\|_{L^q} \|\varphi \|_{L^{p'}}.
  \end{align}
  Since $\domLap$ was topologized by $\|(-\Delta +1)\cdot \|_{L^p}$, we thus find
  \begin{multline}
    \label{eq:domLapBound}
    \sup_{\substack{\varphi \in W^{1,2} \cap L^{p'}, \|\varphi \|_{L^{p'}} \le 1} }\Big |\int_\Omega \rho \nabla \psi \cdot
    \nabla \varphi\dd\x \Big |\\ \leq \Bigl ( \|\rho\|_{L^\infty}\bigl\|\Delta (-\Delta
    +1)^{-1}\bigr\|_{\LL(L^p)} + \mathcal E\bigl(\dom_{L^p}(\Delta), W^{1,q}\bigr) \|\nabla
    \rho\|_{L^q} \Bigr )\|\psi \|_{\dom_{L^p}(\Delta)}
  \end{multline}
  This means that the linear form~\eqref{e-linformlp} is bounded on $(W^{1,2},\|\cdot\|_{L^{p'}})$, such
  that $\psi \in \dom(A_p(\rho))$ by the construction in Chapter~\ref{s-pre}.  Moreover,
  $\|A_p(\rho)\psi\|_{L^p}$ is bounded by the right-hand side in~\eqref{eq:domLapBound}.  The embedding
  $\domLap \embeds \dom(A_p(\rho))$ follows immediately.
  
  ii)~One reasons analogously as in the previous case, but exploits instead of~\eqref{e-eessti} the
  equality
  \[
    \int_\Omega \nabla \psi \cdot \nabla \varphi\dd\x = \int_\Omega \rho^{-1} \rho \nabla \psi \cdot
    \nabla \varphi\dd\x = \int_\Omega \rho \nabla \psi \cdot \nabla (\rho^{-1} \varphi) \dd\x -\int
    _\Omega \varphi \rho \nabla \psi \nabla (\rho^{-1}) \dd\x.
  \]
  This gives $\dom(A_p(\rho)) \embeds \domLap$, from which the Banach space identity
  $\domLap = \dom(A_p(\rho))$ follows.
\end{proof}

\begin{corollary} \label{c-0stetfu} For $p = \frac q2$, the mapping
  \[
    C\bigl(\overline J;W^{1,q}\bigr) \ni \omega \mapsto -\nabla \cdot \omega(\cdot) \nabla
  \]
  takes its values in the space $C\bigl(\overline J;\LL(\dom_{L^p} (\Delta);L^p)\bigr)$ and is
  Lipschitzian on bounded subsets.
\end{corollary}

\section[Investigation of the model]{Investigation of the model}
\label{s-inv}
\subsection{Precise formulation of the problem and main result}
\label{s-formu}
In this section, we give a rigorous analysis of~\eqref{e-u}--\eqref{e-IVs} in the sense of
Definition~\ref{d-1} below. In fact, most of this section will consist of the proof of the main
Theorem~\ref{t-mainres}, which we state in the following. An explanation of the strategy for the proof
can be found in Section~\ref{s-proof}.

Let us first agree on the following: All appearing function spaces are supposed to be \emph{real} ones,
without indicating this explicitly in the sequel.

For all what follows, we suppose Assumption~\ref{a-Koeff} to be satisfied. We moreover fix
$p=\frac{q}{2}$ with $q$ being the number from Assumption~\ref{a-reg}, which is also assumed to hold
true. We abbreviate $A_p(\mu)$ \emph{for this fixed $p$} by $A(\mu)$ for a measurable, bounded and real
coefficient function $\mu$. Fix also a number $r > 2 (1- \frac {d}{q})^{-1}$ and $s >r$.

In the following we want to establish a precise notion of a \emph{solution of the Keller-Segel-Model}.

\begin{definition} \label{d-1} Given a subinterval $I= \iv{0,S}$ of $J$, we call a quadruple of functions
  \begin{equation*}
    \bigl(u,(v,p,w)\bigl) \in \mr^r(I;\domLap,L^p) \times \mr^s(I;\domVecLap,\vecL^p)
  \end{equation*}
  a general solution of~\eqref{e-u}--\eqref{e-IVs} on $I$, if these satisfy
  \begin{align} u'(t)+ A\bigl(\afo(u(t) ,v(t))\bigr) u(t) & = A\bigl(\bfo(u(t), v(t))\bigr) v(t) \notag
    \\ \label{e-preu} & \qquad + R_1\bigl(u(t),v(t),p(t),w(t)\bigr) \\ \label{e-prev} v'(t)- k_v \Delta
    v(t) & = R_2\bigl(u(t),v(t),p(t),w(t)\bigr) \\ 
    p'(t)- k_p \Delta p(t) & = R_3\bigl(u(t),v(t),p(t),w(t)\bigr) \\ \label{e-prew} w'(t)- k_w \Delta
    w(t) & = R_4\bigl(u(t),v(t),p(t),w(t)\bigr) \\ \label{e-anfang} \bigl(u(0),v(0),p(0),w(0)\bigr)
    &=(u_0,v_0,p_0,w_0)
  \end{align}
  for almost all $t \in I$ in $L^p \times \vecL^p$ for~\eqref{e-preu}--\eqref{e-prew}, where the time
  derivative is taken in the sense of vector valued distributions and the initial values satisfy
  \[(u_0,v_0,p_0,w_0) \in \bigl(L^p,\domLap\bigr)_{1-\frac {1}{r},r} \times \bigl((L^p,\domLap)_{1-\frac
      {1}{s},s}\bigr)^3 =: \mathrm{IV}(r,s).\] The operator $-\Delta$ is here to be understood as
  $A_p(1)$, i.e., the restriction of the weak (negative) Laplacian to $L^p$.
\end{definition}

\begin{rem} \label{r-justi1}
  \begin{enumerate}[i)]
  \item In the original model, we had the specific inhomogeneities \begin{align*} R_1(u,v,p,w) & = 0, \\
      R_2(u,v,p,w) &= -r_1vp +r_{-1} w +uf(v),\\ R_3(u,v,p,w) &= -r_1vp +(r_{-1}+r_2)w +u g(v,p),\\
      R_4(u,v,p,w) &= r_1vp -(r_{-1}+r_2)w,
    \end{align*}
    cf.~\eqref{e-u}--\eqref{e-w}. If $f$ and $g$ are continuously differentiable as real functions, this
    choice clearly satisfies the assumptions on the functions $R_i$ as in Assumption~\ref{a-Koeff}.
  \item For almost all $t \in I$ the functions $u(t,\cdot), v(t,\cdot), p(t,\cdot), w(t,\cdot)$ each lie
    in the space $\domLap$, hence for these $t$ a \emph{homogeneous Neumann condition}
    is fulfilled
    in a generalized sense, cf.~Remark~\ref{rem:A2}.
  \item The regularity of the initial values in $\mathrm{IV}(r,s)$ is exactly the optimal one for the
    class of solutions as defined in Definition~\ref{d-1}, cf.\ Remark~\ref{rem:IVpassend}.
  \item Definition~\ref{d-1} is in fact faithful to itself in the sense that the functions and mappings
    indeed map into the correct spaces, see also Remark~\ref{r-justi2} below.
  \end{enumerate}
\end{rem}
We formulate now the main result of this work.

\begin{theorem} \label{t-mainres} Under Assumption~\ref{a-reg}, problem~\eqref{e-u}--\eqref{e-IVs} admits
  exactly one local-in-time general solution in the spirit of Definition~\ref{d-1}. Moreover, the
  solutions $(v,p,w)$ are uniformly bounded in $L^\infty$ over the maximal interval of existence.
\end{theorem}

\begin{rem}
  Considering the derivation of the model in the introductory chapter, the question of
  \emph{positivity} of the solutions $(u,v,p,w)$ in the sense of Definition~\ref{d-1}---provided their
  initial values were positive in the first place---arises naturally. It is a standard result in the
  theory of reaction-diffusion systems (cf.\ e.g.~\cite{Pierre10}) that a system in the
  form~\eqref{e-prev}--\eqref{e-prew} is positivity preserving if and only if the inhomogeneities
  $R_2(\bar u,\cdot),R_3(\bar u,\cdot),R_4(\bar u,\cdot)$ are \emph{quasipositive} for every
  $\bar u \in \R$, that is, if $(\bar v,\bar p,\bar w)$ is an arbitrary vector in $\R^3$ with nonnegative
  entries, then
  \begin{equation*}
    R_2(\bar u,0,\bar p,\bar w) \geq 0, \quad     R_3(\bar u,\bar v,0,\bar w) \geq 0 \quad \text{and} \quad    R_4(\bar u,\bar v,\bar p,0) \geq 0.
  \end{equation*}
  The specific inhomogeneities in~\eqref{e-v}--\eqref{e-w}, cf.~Remark~\ref{r-justi1}, indeed satisfy
  this condition if $g(\bar v,0) \geq 0$ and $f(\bar u) \geq 0$ for nonnegative $\bar v,\bar u \geq
  0$. Hence,~\eqref{e-prev}--\eqref{e-prew} is positivity preserving for $(v,p,w)$ if $u$ is also a
  positive function, i.e.,~\eqref{e-preu} is \emph{also} positivity preserving. Unfortunately, the latter
  seems very difficult to show in the very general context of Definition~\ref{d-1}, even with $R_1 = 0$,
  and is generally not true for seemingly easy cases, see~\cite[Ch.~5]{Nagel89}. However, for the
  specific choices $\afo(u,v) = 1$ and $\bfo(u,v) = -u$, already mentioned in the introduction as
  well-researched model choices, positivity of $u$ is shown in~\cite[Thm.~3.3]{Gajewski:1998}
  \emph{independent} of the sign of $v$. The proof in~\cite{Gajewski:1998} only relies on the fact that
  $v$ is uniformly bounded in time and space, which is the case for our solutions obtained from
  Theorem~\ref{t-mainres}. Hence, for this choice of $\afo$ and $\bfo$, $R_1 = 0$ and
  $R_2(\bar u,\cdot),R_3(\bar u,\cdot),R_4(\bar u,\cdot)$ quasipositive for $\bar u \geq 0$,
  system~\eqref{e-preu}--\eqref{e-prew} is indeed positivity preserving. This includes in particular
  system~\eqref{e-u}--\eqref{e-w} for this choice of $\afo$ and $\bfo$ and $f,g$ as mentioned above.
\end{rem}

We now proceed with the proof of the main result.

\subsection{The proof}\label{s-proof}
The actual proof of Theorem~\ref{t-mainres} works in as follows. It should be evident to the reader that
we plan to use the abstract result of Amann, Theorem~\ref{t-Amann}. The general idea is to solve the
semilinear equations for $(v,p,w)$,~\eqref{e-prev}--\eqref{e-prew}, in dependence of $u$, and to show
that this dependence re-inserted in the first equation for $u$ satisfies the assumptions in
Theorem~\ref{t-Amann}. Here, it is clear that the dependence of $(v,p,w)$ on $u$ will be \emph{nonlocal
  in time}, which indeed makes Theorem~\ref{t-Amann}---instead of other well-known abstract quasilinear
existence results---necessary.

However, as~\eqref{e-prev}--\eqref{e-prew} are \emph{nonlinear} equations themselves, it is not \emph{a
  priori} clear that they in fact admit global solutions on the whole time horizon $\iv{0,T}$, and a
local-in-time existence interval $I(u)$ for $(v,p,w)$ depending on $u$ would clearly thwart any attempt
to establish the assumptions from Theorem~\ref{t-Amann}. Hence, we modify the right-hand sides
in~\eqref{e-prev}--\eqref{e-prew} by introducing a suitable cut-off, which then allows to show global
existence, uniqueness, and a well-behaved dependence on $u$ for the solutions $(\hat v,\hat p,\hat w)$ of
the modified lower system (\eqref{e-precv}--\eqref{e-precw} below); this is Theorem~\ref{t-hilfs}.

After establishing that the involved operators and functions satisfy the assumptions of
Theorem~\ref{t-Amann}, we then use that very theorem to show existence and uniqueness of a local-in-time
solution $u$ to the \emph{modified} system, including the equation for $u$, in Theorems~\ref{l-loes0}
and~\ref{l-modisystem}. From there, we finally obtain Theorem~\ref{t-mainres} by showing that the
local-in-time solution obtained for the modified system is indeed also the solution to the original
system~\eqref{e-preu}--\eqref{e-anfang} at the cost of a possibly still smaller existence interval.

Aside from the dependence of $(v,p,w)$ on $u$, there is another major obstacle when working to satisfy
the assumptions of Theorem~\ref{t-Amann}: Assumption~\ref{t-Amann-evil}) of said theorem in fact
requires, in our notation, that the differential operators, which will be
$A\bigl(\kappa(u(t),v(u)(t))\bigr)$, have \emph{uniform} domains $Y$ for \emph{all}
$u \in \mr^r(J;Y,L^p)$ and for almost every $t \in J$. Thanks to Lemma~\ref{l-domAIN}, we will be able to
use $Y = \domLap$, provided that the coefficient functions $\kappa(u(t),v(u)(t))$ are from $W^{1,q}$ for
almost every $t \in J$. We have already laid the foundations to show this in Lemma~\ref{l-carryover},
together with the maximal regularity embedding~\eqref{e-embedcont}, which together immediately yield the
following introductory result which is of importance in all what follows.

\begin{lemma} \label{l-embeddd} Set $\alpha =\frac {1}{2} -\frac {d}{2q} - \frac {1}{r}$. By the choice
  of $r$, we have $\alpha > 0$.
  \begin{enumerate}[i)] \item The space $\mr^r(J;\dom_{L^p}(\Delta),L^p)$ embeds into
    $C^\alpha(J;W^{1,q})$ and, hence, compactly into $C(\overline J;C(\overline \Omega))$.
  \item Analogously, $\mr^s_0(J;\dom_{\vecL^p}(\Delta),\vecL^p)$ embeds into
    $C^\alpha(J;\mathbb W^{1,q})$ and, hence, compactly into $C(\overline J;C(\overline \Omega)^3)$.
  \end{enumerate}
\end{lemma}
\begin{proof}The compactness in both cases follows by the vector-valued Arzel\`{a}-Ascoli theorem,
  cf.~\cite[Ch.~III.3]{lang}. For i),~the condition on $r$ implies
  $1-\frac {1}{r} -\bigl (\frac {1}{2}+ \frac {d}{2q}) >0$.  Thus, the claim follows from
  Lemma~\ref{l-maxparregfacts}, cf.~\eqref{e-maxreghoelderembed}, in conjunction with
  Lemma~\ref{l-carryover}.  ii)~is proved analogously.
\end{proof}

\begin{rem} \label{r-justi2} For $u \in \mr^r(I;\dom_{L^p}(\Delta),L^p)$ and
  $v \in \mr^s(I;\dom_{L^p}(\Delta),L^p)$ with $I$ as in Definition~\ref{d-1}, Lemma~\ref{l-embeddd}
  together with Lemma~\ref{l-domAIN} and the assumptions on $\afo$ and $\bfo$ (cf.\ Assumption~\ref{a-reg}) tells us that $\afo(u(t),v(t))$ and $\bfo(u(t),v(t))$ are each functions from $W^{1,q}$
  for every $t \in \overline I$. Together with $u(t),v(t) \in \domLap$ for almost every $t \in I$, this
  shows that the expressions $A(\afo(u(t) ,v(t))) u(t)$ and $A(\bfo(u(t), v(t))) v(t)$ in~\eqref{e-preu}
  are indeed well-defined. See also Lemmata~\ref{l-Lipschitzconty} and~\ref{l-F} below.
\end{rem}

We will now modify the abstract system~\eqref{e-preu}--\eqref{e-prew} in such a way that the terms on the
right hand sides of~\eqref{e-prev}--\eqref{e-prew} become bounded in space and time. This will
ultimately lead to a solution in the spirit of Definition~\ref{d-1} on a \emph{smaller} time interval,
since the modification becomes ``active'', only after some time point $T_\bullet >0$, allowing to
re-obtain the correct solution to the unmodified system on $[0,T_\bullet]$.

We consider
\begin{equation}
  (v_0,p_0,w_0) \in \bigl((L^p,\domLap)_{1-\frac
    {1}{s},s}\bigr)^3\label{e-IVreg}
\end{equation}
to be given and fixed from now on.

\begin{definition} \label{d-2} For $\delta >0$, we put
  $M:= \delta+ \max( \|v_0\|_{L^\infty}, \|p_0\|_{L^\infty},\|w_0\|_{L^\infty})$.  Let
  $\eta \in C^\infty(\R)$ be a smooth function which is the identity on the interval $[-M,M]$ and is
  equal to $-(M+1)$ on the interval $\mathopen]-\infty, -(M+1)]$ and equal to $M+1$ on the interval
  $[M+1,\infty\mathclose[$.  Moreover, we put
  $R^{\eta}_i:=R_i(\cdot, \eta(\cdot),\eta(\cdot),\eta(\cdot))$ for $i = 2,3,4$.
\end{definition}

Note that, due to Lemma~\ref{l-carryover} and the choice of $s$, we have the embedding 
$(L^p,\domLap)_{1-\frac{1}{s},s} \embeds C(\overline \Omega)$, such that the number $M$ in
Definition~\ref{d-2} is well-defined.

We further split off the initial values for the functions $v,p,w$ for which we put
$v_\mathcal I(t)=\sg{t\,k_v \Delta}v_0$ as well as $p_\mathcal I(t)=\sg{t\,k_p \Delta}p_0$ and
$w_\mathcal I(t)=\sg{t\,k_w \Delta}w_0$, and write
\begin{equation} \label{e-splitoff} v=v_\mathcal I +\check v,\quad p=p_\mathcal I +\check p,\quad
  w=w_\mathcal I + \check w,
\end{equation}
where $\check v, \check p$ and $\check w$ have the initial value $0$, of course.

For convenience, we collect some of the properties for the functions $v_{\mathcal{I}}, p_{\mathcal{I}}$
and $w_{\mathcal{I}}$ which will be of importance later.
\begin{lemma} \label{l-innisplit} Let the initial values $(v_0,p_0,w_0)$ satisfy~\eqref{e-IVreg}.
  \begin{enumerate}[i)]
  \item One has
    \begin{equation} \label{e-inhomanf} v_{\mathcal{I}}' - k_v \Delta v_{\mathcal{I}} = p_{\mathcal{I}}'
      - k_p \Delta p_{\mathcal{I}} = w_{\mathcal{I}}' - k_w \Delta w_{\mathcal{I}} \equiv 0
    \end{equation}
    on any time interval $]0,S[\ \subseteq J$.
  \item The functions $v_{\mathcal{I}}, p_{\mathcal{I}}$ and $w_{\mathcal{I}}$ are each from
    $\mr^s(J;\domLap,L^p)$, take their values pointwise on $J$ in $W^{1,q}$, and are continuous on every
    time interval $[0,S[\ \subset \overline{J}$.
  \item The functions $v_{\mathcal{I}}, p_{\mathcal{I}}$ and $w_{\mathcal{I}}$ are continuous on every
    time interval $[0,S[\ \subset \overline{J}$ when considered as $C(\overline
    \Omega)$-valued. Moreover, in this case we have
    \begin{multline*}
      \|v_{\mathcal{I}}(t)\|_{C(\overline \Omega)} \le \|v_{\mathcal{I}}(0)\|_{C(\overline \Omega)},
      \quad \|p_{\mathcal{I}}(t)\|_{C(\overline \Omega)} \le \|p_{\mathcal{I}}(0)\|_{C(\overline
        \Omega)}, \\ \text{and} \quad \|w_{\mathcal{I}}(t)\|_{C(\overline \Omega)} \le
      \|w_{\mathcal{I}}(0)\|_{C(\overline \Omega)}
    \end{multline*}
    for every $s \in S$
  \end{enumerate}
\end{lemma}
\begin{proof}
  i)~is clear. ii)~Lemma~\ref{l-maxparregfacts}~ii) shows that the functions
  $v_{\mathcal{I}}, p_{\mathcal{I}}, w_{\mathcal{I}} $ are continuous when considered as
  $(L^p,\domLap)_{1-\frac {1}{s},s}$-valued ones. Thus, the assertion follows from
  Lemma~\ref{l-carryover} and the definition of $s$.  iii)~The first assertion follows from~ii) by
  embedding $W^{1,q} \embeds C(\overline \Omega)$.  Moreover, since the semigroups act as
  \emph{contractive} ones in $L^\infty$, cf.\ Proposition~\ref{prop:collection-Ap}, the evolution of the
  initial values $v_0, p_0, w_0$ does not lead to larger $L^\infty$-norms. The latter is identical with
  the $C(\overline \Omega)$-norm in our case.
\end{proof}

Having introduced the modified nonlinearities $R^\eta_i$ and the split-off of the initial values, we
combine both into the functions $\S_i \colon J \times C(\overline\Omega) \times \vecL^p \to L^p$ by
\begin{equation*} 
  \S_i(t;\mathfrak u,\mathfrak v,\mathfrak p,\mathfrak w) :=
  R^\eta_i\bigl(\mathfrak u,v_{\mathcal{I}}(t) + \mathfrak
  v,p_{\mathcal{I}}(t) + \mathfrak p,w_{\mathcal{I}}(t) + \mathfrak w \bigr)
\end{equation*}
for $i = 2,3,4$, and
\[\S_1(t;\mathfrak u,\mathfrak v,\mathfrak p,\mathfrak w) :=
  R_1\bigl(\mathfrak u,v_{\mathcal{I}}(t) + \mathfrak v,p_{\mathcal{I}}(t) + \mathfrak
  p,w_{\mathcal{I}}(t) + \mathfrak w \bigr).\] Then we consider instead of~\eqref{e-preu}--\eqref{e-anfang} the system
\begin{align}
  u'(t)+ A\bigl(\afo(u(t) ,v_{\mathcal{I}}(t) + v(t))\bigr) u(t) & =
  A\bigl(\bfo(u(t),v_{\mathcal{I}}(t) + v(t)\bigr)\bigl(v_{\mathcal{I}}(t) +
  v(t)\bigr) \notag \\ & \qquad
  + \S_1\bigl(t;u(t),v(t),p(t),w(t)\bigr), \label{e-precu}\\
  \label{e-precv}
  v'(t)- k_v \Delta v(t) & = \S_2\bigl(t;u(t),v(t),p(t),w(t)\bigr), \\
  p'(t)- k_p \Delta p(t) & = \S_3\bigl(t;u(t),v(t),p(t),w(t)\bigr), \\
  \label{e-precw}
  w'(t)- k_w \Delta w(t) & = \S_4\bigl(t;u(t),v(t),p(t),w(t)\bigr), \\
  \label{e-precIV}
  (u(0),v(0),p(0),w(0)) & = (u_0,0,0,0)
\end{align}
as equations in the Banach space $L^p \times \vecL^p \times \mathrm{IV}(r,s)$, holding for almost every
$t \in I$ for the first four components. Note that we have, by abuse of notation, returned to writing
$v,p$ and $w$ instead of $\check{v},\check{p}$ and $\check{w}$ as introduced in~\eqref{e-splitoff} for
better readability. Since we work exclusively with the functions with initial value $0$ from here on,
this should not give rise to confusion to the reader.

After these preparations we prove the subsequent theorem, from which our main result,
Theorem~\ref{t-mainres}, then follows (and which is in fact only a slight reformulation of this).
\begin{theorem} \label{l-modisystem} For given $(u_0,v_0,p_0,w_0) \in \mathrm{IV}(r,s)$, the system~\eqref{e-precu}--\eqref{e-precIV} admits exactly one local-in-time solution
  \[\bigl(u,(v,p,w)\bigr) \in \mr^r(I;\dom_{ L^p}(\Delta), L^p) \times \mr^s_0(I;\dom_{\mathbb
      L^p}(\Delta),\vecL^p),\] with $I = \iv{0,S} \subseteq J$.
\end{theorem}

Let us re-iterate the strategy for the proof of Theorem~\ref{l-modisystem}: Firstly, we will solve the
equations~\eqref{e-precv}--\eqref{e-precw} with $u \in C(\overline J;C(\overline \Omega))$ fixed by a
fixed-point argument. The crucial point is that the dependence of these solution $(v,p,w)$ from $u$ is
well-behaved in the space $\mr^s_0(J;\domVecLap,\vecL^p)$. So implicitly inserting this
into~\eqref{e-precu}, this equation decouples from the other ones and is tractable by means of Amann's
result, Theorem~\ref{t-Amann}. Having then $u$ at hand (we prove that the assumptions of
Theorem~\ref{t-Amann} are satisfied in Theorem~\ref{l-loes0}), one ``re-discovers'' $(v,p,w)$
by~\eqref{e-precv}--\eqref{e-precw}.

\begin{theorem} \label{t-hilfs}
  \begin{enumerate}[i)]
  \item \label{t-hilfs-exist} Assume $u \in C(\overline J;C(\overline \Omega))$ to be given. Then the
    system~\eqref{e-precv}--\eqref{e-precw} has a unique solution
    $(v,p,w) \in \mr^s_0(J;\dom_{\mathbb L^p}(\Delta),\vecL^p)$.
  \item \label{t-hilfs-diff} Let
    $\sol\colon C(\overline J; C(\overline \Omega)) \to \mr^s_0(J;\domVecLap,\vecL^p)$ denote the mapping which
    assigns to $u$ the solution of~\eqref{e-precv}--\eqref{e-precw}. Then $\sol$ is continuously
    differentiable.
  \end{enumerate}
\end{theorem}

\begin{proof}
  \ref{t-hilfs-exist}):~For given $u \in C(\overline J;C(\overline \Omega))$, define
  $R_{i,u}\colon J \times \vecL^p \to L^p$, $i=2,3,4$ by setting $R_{i,u}(t;v,p,w):=\S_i(t;u(t),v,p,w)$.
  It is not hard to see that each $R_{i,u}$ is uniformly continuous on
  $J$ when the second argument $(v,p,w) \in
  \vecL^p$ is fixed, and globally Lipschitz continuous on $\vecL^p$ when $t \in
  J$ is fixed -- with a Lipschitz constant uniform in
  $t$. Therefore, the \emph{semilinear} parabolic system~\eqref{e-precv}--\eqref{e-precw}
  admits exactly one \emph{mild} solution $(\hat v, \hat p, \hat w) \in C(\overline
  J;\vecL^p)$ with initial value zero, cf.~\cite[Ch.~6, Thm.~1.2]{pazy}.  Since then the mapping
  \[
    J \ni t \mapsto \Bigl(R_{2,u}\bigl(t;\hat v(t), \hat p(t),\hat w(t)\bigr), R_{3,u}\bigl(t;\hat v(t), \hat p(t),\hat
    w(t)\bigr), R_{4,u}\bigl(t;\hat v(t), \hat p(t),\hat w(t)\bigr)\Bigr)
  \]
  belongs to $L^\infty(J;\vecL^p)$, maximal parabolic regularity of the operator
  \[
    -\widetilde \Delta:= \diag(-k_v \Delta,-k_p \Delta, -k_p \Delta)
  \]
  on $\vecL^p$ provides an unique solution $(\check v, \check p, \check w)$ with zero initial values of the equations
  \begin{align*} 
    v'(t)-k_v \Delta v(t) &= R_{2,u}\bigl(t;\hat v(t), \hat p(t),\hat w(t)\bigr), 
    \\ 
    p'(t)-k_p \Delta p(t) &= R_{3,u}\bigl(t;\hat v(t), \hat p(t),\hat w(t)\bigr), 
    \\ 
    w'(t)-k_w \Delta w (t) &= R_{4,u}\bigl(t;\hat v(t), \hat p(t),\hat w(t)\bigr), 
  \end{align*}
  which even belongs to the space $\mr^s_0(J;\domVecLap,\vecL^p)$.  But this solution $(\check v, \check
  p, \check
  w)$ is also a mild solution of~\eqref{e-precv}--\eqref{e-precw}, cf.~\cite[Ch.~III.1.3]{A95}. Since
  then both $(\hat v, \hat p, \hat w)$ and $(\check v, \check p, \check
  w)$ are mild solutions of~\eqref{e-precv}--\eqref{e-precw} with the same initial value, they must
  necessarily coincide. Hence, $(\hat v, \hat p, \hat w)$ belongs to $ \mr^s_0(J;\domVecLap,\vecL^p)
  $ and is the unique function to solve~\eqref{e-precv}--\eqref{e-precw}.\\
  \ref{t-hilfs-diff})~For this we apply the implicit function theorem, considering the mapping
  \[
    \Psi \colon C(\overline J; C(\overline \Omega))\times \mr^s_0(J; \domVecLap,\vecL^p) \to L^s(J;
    \vecL^p),
  \]
  which is given by
  \begin{align*}
    \Psi(u,v,p,w)(t) & =
    \Bigl (v'(t)- k_v \Delta v(t) -R_{2,u}\bigl(t; v(t),  p(t), w(t)\bigr), \\
    & \qquad p'(t)- k_p \Delta p(t) - R_{3,u}\bigl(t;  v(t), p(t), w(t)\bigr), \\
    & \qquad w'(t)- k_w \Delta w(t) - R_{4,u} \bigl(t; v(t), p(t), w(t)\bigr) \Bigr)
  \end{align*}
  Obviously, for given $u \in C(\overline J;C(\overline \Omega))$, the triple $(v,p,w) \in
  \mr^s_0(J;\domVecLap,\vecL^p)$ is a solution of~\eqref{e-precv}--\eqref{e-precw} iff $\Psi(u,v,p,w) =
  0$ in $L^s(J;\vecL^p)$. By the assumptions on $R_2,R_3$ and $R_4$,
  $\Psi$ is continuously differentiable and the partial derivative with respect to the second variable in
  a given point $\bigl (\bar u,(\bar v, \bar p, \bar w)\bigr ) \in C(\overline J;C(\overline \Omega))
  \times
  \mr^s_0(J;\domVecLap,\vecL^p)$ is the linear mapping which assigns to the triple $(h_2,h_3,h_4)\in
  \mr^s_0(J;\domVecLap,\vecL^p)$ the expression
  \begin{align}
    &\left[\left(\partial_{(2,3,4)}\Psi\right)(\bar u,\bar v,\bar p,\bar w)(h_2,h_3,h_4)\right](t)
    \notag \\ & \qquad = \left[h_2'(t)- k_v \Delta h_2(t) -
      \sum_{i=2}^4 \partial_i R_{2,u}\bigl(t;\bar v(t),\bar p(t),\bar w(t)\bigr)h_i(t)
      ,\right.\label{e-ableitg1} \\ &
    \qquad \qquad h_3'(t)- k_v \Delta h_3(t) -
    \sum_{i=2}^4 \partial_i R_{3,u}\bigl(t;\bar v(t),\bar p(t),\bar w(t)\bigr)h_i(t),
    \\ &
    \qquad \qquad \left.h_4'(t)- k_v \Delta h_4(t) -
      \sum_{i=2}^4 \partial_i R_{4,u}\bigl(t;\bar v(t),\bar p(t),\bar w(t)\bigr)h_i(t)
    \right],  \label{e-ableitg2}
  \end{align}
  which is a function from $L^s(J;\vecL^p)$. We know already that the operator
  $-\widetilde\Delta$ satisfies maximal parabolic regularity on the space
  $\vecL^p$. Moreover, it is clear that the remaining terms in front of the directions
  $h_i$ in~\eqref{e-ableitg1}--\eqref{e-ableitg2}, considered as time-dependent multipliers on the
  corresponding $L^p$-space, form \emph{bounded} operators in
  $L^s(J;\vecL^p)$, since the corresponding multipliers are bounded and continuous in space and
  time. Hence, according to a suitable perturbation theorem as in~\cite[Prop.~1.3]{arendt1}, the equation
  \[\left(\partial_{(2,3,4)}\Psi\right)(\bar u,\bar v,\bar p,\bar
    w)(h_2,h_3,h_4) = \mathfrak f\] is uniquely solvable for every $\mathfrak f \in
  L^s(J;\vecL^p)$ with $(h_2,h_3,h_4)\in
  \mr^s_0(J;\domVecLap,\vecL^p)$. This means that the partial derivative
  $\left(\partial_{(2,3,4)}\Psi\right)(\bar u,\bar v,\bar p,\bar
  w)$ is a topological isomorphism between $\mr^s_0(J;\domVecLap,\vecL^p)$ and
  $L^s(J;\vecL^p)$, what makes the implicit function theorem applicable. Considering $\Psi(\bar u,\bar
  v,\bar p,\bar w) = 0$ and $(\bar u,\bar v,\bar p,\bar w) = (\bar u,\sol(\bar
  u))$, we thus obtain that the implicit function defined on a neighborhood of $\bar
  u$, whose existence is guaranteed by the implicit function theorem, coincides with
  $\sol$ on that neighborhood and is continuously differentiable. Since this is true for \emph{every}
  function $\bar u \in C(\overline J;C(\overline \Omega))$, the ``solution operator''
  $\sol$ is continuously differentiable on that space.
\end{proof}

\begin{rem} \label{r-uniform} In addition to the results of Theorem~\ref{t-hilfs}, the above
  considerations make it clear that the set of solutions $\{ \sol(u) \colon u \in
  \mathfrak{B}\}$ which corresponds to a \emph{bounded} subset $\mathfrak{B}$ of $C(\overline
  J;C(\overline
  \Omega))$ in turn forms a \emph{bounded} subset in the space
  $\mr^s_0(J;\domVecLap,\vecL^p)$, and, hence, a precompact one in $C_0(\overline J;C(\overline
  \Omega)^3)$, cf.\ Lemma~\ref{l-embeddd}. This can be seen by observing that the real functions
  $R_{i,u}$, $i =
  2,3,4$, acting as right hand sides in~\eqref{e-precv}--\eqref{e-precw} are uniformly bounded in
  $L^s(J;\vecL^p)$ in the following way: We set
  \begin{equation*}
    M_R := \max_i M_{i,R} < \infty, \quad \text{where} \quad M_{i,R} := \sup_{\substack{|\bar u| \leq M_{\mathfrak B},\\
        |\bar v| \vee |\bar p| \vee |\bar w| \leq M+1}} \bigl|R_i(\bar u,\bar v,\bar p,\bar w)\bigr|,
  \end{equation*}
  using $M_{\mathfrak B} := \max_{u \in \mathfrak B} \|u\|_{C(\overline J;C(\overline\Omega))}$. Then
  \begin{equation*}
    \max_{i}\,\bigl\|R_{i,u}\bigl(\cdot\,;v(\cdot),p(\cdot),w(\cdot)\bigr)\bigr\|_{L^\infty(J;\vecL^p)}  \leq |\Omega|^{\frac1p} M_R
  \end{equation*}
  for \emph{all} $u \in \mathfrak B$ and $(v,p,w) \in
  \mr_0^s(J;\domVecLap,\vecL^p))$, which by the maximal parabolic regularity
  estimate~\eqref{eq:MPRestimate} shows that $\{ \sol(u) \colon u \in
  \mathfrak{B}\}$ forms a bounded set in the space $\mr^s_0(J;\domVecLap,\vecL^p)$.
\end{rem}

Out next intention is to show that the mapping $\sol$ is Lipschitzian on bounded subsets of
$\mr^r(J;\domLap,L^p)$.

\begin{corollary} \label{c-Lipschitzian} Let $\BB$ be any bounded subset of
  $\mr^r(J;\dom_{ L^p}(\Delta), L^p)$. Then the mapping $\sol$ is Lipschitzian as a mapping from $\BB$
  into $\mr^s_0(J;\domVecLap,\vecL^p)$, and hence, also into $C(\overline J;\mathbb W^{1,q})$.
\end{corollary}
\begin{proof}
  Without loss of generality we may assume that $\BB$ is a---sufficiently large---ball.  Any bounded
  subset $\BB$ of $\mr^s_0(J;\domVecLap,\vecL^p)$ forms a precompact subset of
  $C(\overline J;C(\overline \Omega))$, according to Lemma~\ref{l-embeddd}.  Accordingly, its closure
  $\overline {\BB}$ in $C(\overline J;C(\overline \Omega))$ forms a compact set in this space which is
  convex, too. Now Theorem~\ref{t-hilfs}~\eqref{t-hilfs-diff} tells us that the derivative of
  $\mathcal S$ is bounded on $\overline {\BB}$. Since this set contains with any two points also the
  segment between them, an application of the mean value theorem gives the first claim.  Finally, the
  assertion for $C(\overline J;\mathbb W^{1,q})$ is obtained from the previous one via
  Lemma~\ref{l-embeddd}.
\end{proof}

Having introduced the solution operator $\sol$ for~\eqref{e-precv}--\eqref{e-precw}, we now turn back to
Theorem~\ref{l-modisystem}. Inserting $\sol(u)$ with $u \in \mr^r(J;\dom_{L^p}(\Delta), L^p)$ for
$(v,p,w)$ in~\eqref{e-precu}, one obtains a self-consistent equation for $u$ alone together with the
initial value condition $u(0) = u_0$. This equation can be solved via Theorem~\ref{t-Amann}, as we will
show below. Afterwards, having the solution $\bar u$ at hand, the functions $(\bar v, \bar p, \bar w)$
are determined via Lemma~\ref{t-hilfs} or $\sol(\bar u)$, from which they
satisfy~\eqref{e-precv}--\eqref{e-precw} automatically by construction. The quality of the whole solution
of~\eqref{e-precu}--\eqref{e-precw} is then $\bar u \in \mr^r(J;\dom_{L^p}(\Delta), L^p)$ and
$(\bar v,\bar p,\bar w) \in \mr_0^s(I;\dom_{\vecL^p}(\Delta),\vecL^p)$.

We have formulated the next big step---the application of Theorem~\ref{t-Amann}---as a theorem on its
own. For this, let $\sol_1$ denote the $v$-component of $\sol$, $\sol_2$ the $p$-component of $\sol$, and
$\sol_3$ the $w$-component of $\sol$.

\begin{theorem} \label{l-loes0} Suppose $(u_0,v_0,p_0,w_0) \in \mathrm{IV}(r,s)$. 
  Then there exists a maximal interval $I_\bullet= \iv{0,S_\bullet} \subseteq J$ such that the equation
  \begin{multline}
    u'(t)+ A\Bigl(\afo\bigl(u(t) ,v_{\mathcal{I}}(t) + \sol_1(u)(t)\bigr)\Bigr) u(t) \\ =
    A\Bigl(\bfo\bigl(u(t),v_{\mathcal{I}}(t) + \sol_1(u)(t)\bigr)\Bigr)\bigl(v_{\mathcal{I}}(t) + \sol_1(u)(t)\bigr)
    + \S_1\bigl(t;u(t),\sol(u)(t)\bigr),
    \label{e-0u}
  \end{multline}
  has a unique solution $u \in \mr^r(I;\dom_{L^p}(\Delta), L^p)$ with initial value $u(0) = u_0$ on every
  subinterval $I = \iv{0,S} \subset I_\bullet$.
\end{theorem}

In order to validate the suppositions in Theorem~\ref{t-Amann}, we will formulate some lemmata:

\begin{lemma} \label{l-Lii;ps} Let $\cfo \colon \R^2 \to \R$ be twice continuously differentiable.
  Then the superposition operator $(\psi, \varphi) \to \cfo(\psi(\cdot),\varphi(\cdot))$ induced by
  $\cfo$ is well defined and Lipschitzian on bounded sets when considered as an operator from
  $W^{1,q} \times W^{1,q}$ into $W^{1,q}$.
\end{lemma}

\begin{proof}
  Let $\BB$ be a bounded set in $W^{1,q}$ and assume firstly that
  $\psi , \varphi \in \BB \cap C^\infty(\Omega)$.  Taking into account that $\BB$ forms a bounded subset
  of $C(\overline \Omega)$, a straight forward calculation shows the existence of a constant
  $c=c(\BB,\cfo)$ such that
  \begin{equation} \label{e-lipschitt} \bigl\|\cfo(\psi_1,\varphi_1) -\cfo(\psi_2,\varphi_2)\bigr\|_{W^{1,q}} \le c
    \bigl(\|\psi_1 -\psi _2\|_{W^{1,q}} + \|\varphi_1 -\varphi _2\|_{W^{1,q}}\bigr),
  \end{equation}
  holds for all $ \psi , \varphi \in \BB \cap C^\infty(\Omega)$ . Thus, the superposition operator
  induced by $\cfo$ is defined on a dense subset of $\BB \times \BB \subset W^{1,q}\times W^{1,q}$ and is
  uniformly continuous in $W^{1,q}$ w.r.t.\ the $W^{1,q} \times W^{1,q}$-topology. Hence, it can be
  extended to all of $\BB \times \BB$, with the same estimate as in~\eqref{e-lipschitt}.
\end{proof}

We immediately obtain the following extension from the preceding lemma.

\begin{corollary} \label{c-lipsdchd} Let $\cfo \colon \R^2 \to \R$ be twice continuously
  differentiable. In the spirit of Lemma~\ref{l-Lii;ps}, $\cfo$ induces a superposition operator
  $C(\overline J;W^{1,q}) \times C(\overline J;W^{1,q}) \to C(\overline J;W^{1,q})$ via
  \begin{equation*}
    C\bigl(\overline J;W^{1,q}\bigr) \times C\bigl(\overline J;W^{1,q}\bigr) \ni (\psi,\varphi) \mapsto
    \Bigl[t\mapsto \cfo\bigl(\psi(t),\varphi(t)\bigr)\Bigr] \in C\bigl(\overline J;W^{1,q}\bigr),
  \end{equation*}
  and this mapping is also Lipschitzian on bounded sets.
\end{corollary}

The next lemma covers the differential operators occurring in~\eqref{e-precu}.

\begin{lemma} \label{l-Lipschitzconty} Let $\cfo \colon \R^2 \to \R$ be twice continuously
  differentiable.
  \begin{enumerate}[i)]
  \item The operator
    \begin{equation} \label{e-mathcalA} \mathcal A(u)(t) := A\Bigl(\cfo \bigl(u(t),v_{\mathcal{I}}(t) +
      \sol_1(u)(t)\bigr)\Bigr)
    \end{equation}
    defines a mapping
    \begin{equation*} 
      \mathcal
      A\colon\mr^r(J;\domLap,L^p) \to C\bigl(\overline
      J;\LL(\domLap;L^p)\bigr).
    \end{equation*}
    Moreover, $\mathcal A$ is Lipschitzian on bounded subsets of $\mr^r(J;\domLap,L^p)$.
  \item If, additionally, $\cfo$ is a strictly positive function, then $\mathcal A(u)|_I$ provides a
    topological isomorphism between $\mr_0^r(I;\domLap, L^p)$ and $L^r(I;L^p)$ for every subinterval
    $I = \iv{0,S} \subseteq J$ and every $u\in\mr^r(J;\domLap,L^p)$. In particular, $\mathcal A$
    satisfies assumptions~i) and~ii) in Theorem~\ref{t-Amann} for the spaces $X = L^p$ and $Y = \domLap$
    in this case.
  \end{enumerate}
\end{lemma}
\begin{proof}
  i)~According to Lemma~\ref{l-embeddd}, both spaces $\mr^r(J;\domLap,L^p)$ and 
  $\mr_0^s(\domLap,L^p)$ each embed continuously into $C(\overline J;W^{1,q})$.
  Hence, both $u$ and $\sol_1(u)$ are from $C(\overline J;W^{1,q})$, cf.\ Theorem~\ref{t-hilfs}. Due to
  to Lemma~\ref{l-innisplit} and~\eqref{e-IVreg}, this is also true for the function
  $v_{\mathcal{I}}(\cdot)$.  Thanks to Corollary~\ref{c-lipsdchd}, then the function
  $\cfo \bigl(u(\cdot),v_{\mathcal{I}}(\cdot) +\sol_1(u)(\cdot)\bigr)$ is also from
  $C(\overline J;W^{1,q})$. This allows to apply Corollary~\ref{c-0stetfu}, which shows that $\mathcal A$
  as given in~\eqref{e-mathcalA}, is well-defined as a mapping into the space
  $C(\overline J;\LL(\domLap;L^p))$.

  Let us further show the Lipschitz continuity of $\mathcal A$ on bounded subsets of the space
  $\mr^r(J;\domLap,L^p)$. Combining Corollary~\ref{c-Lipschitzian} and Lemma~\ref{l-Lii;ps} shows that
  the mapping
  \[
    \mr^r(J;\domLap,L^p) \ni u \mapsto \cfo \bigl(u(\cdot),v_{\mathcal{I}}(\cdot) +
    \sol_1(u)(\cdot)\bigr) \in C(\overline J;W^{1,q})
  \]
  is well-defined and Lipschitzian on bounded subset of $\mr^r(J;\domLap,L^p)$. Now it remains to apply
  Corollary~\ref{c-0stetfu}.

  ii)~Clearly, assumption~i) of Theorem~\ref{t-Amann} is already covered by the first assertion in this
  lemma. Let $u$ be a fixed function from $\mr^r(J;\domLap,L^p)$.  Under the positivity condition on
  $\cfo$, the functions $ \cfo\bigl (u(t),v_{\mathcal{I}}(\cdot) + \sol_1(u)(t)\bigr) \in W^{1,q}$ are
  measurable and bounded from above and below by positive constants, uniformly for all
  $t \in \overline J$. Thus, the operators $\mathcal A(u)(t)$ satisfy maximal parabolic regularity on
  $L^p$ for each fixed $t \in J$, cf.\ Theorem~\ref{t-unsere}. Moreover, $t \mapsto \mathcal A(u)(t)$
  belongs to $C(\overline I;\LL(\domLap;L^p))$ for every subinterval $I = \iv{0,S} \subseteq J$
  by~i). But then Theorem~\ref{t-pruessschnaub} tells us that the non-autonomous operator $\mathcal A(u)$
  on every such $I$ satisfies maximal parabolic $L^r(I;\domLap,L^p)$-regularity. This is exactly
  assumption~ii) in Theorem~\ref{t-Amann}.
\end{proof}

Let us now turn to the right-hand side in~\eqref{e-precu}.

\begin{lemma} \label{l-F} Define for $u \in \mr^r(J;\dom_{L^p}(\Delta),L^p)$ the following operators:
  \begin{align}
    \label{e-f1} F_1(u) & := A\Bigl(\bfo
    \bigl(u(\cdot),v_{\mathcal{I}}(\cdot) +\sol_1(u)(\cdot)\bigr)\Bigr ) v_{\mathcal{I}}(\cdot), \\
    F_2(u) & :=  A\Bigl(\bfo
    \bigl(u(\cdot),v_{\mathcal{I}}(\cdot) +\sol_1(u)(\cdot)\bigr)\Bigr ) \bigl[\sol_1( u) (\cdot)\bigr], \\
    \label{e-f3}
    F_3(u) & :=\S_1\bigl(\cdot\,;u(\cdot),\sol(u)(\cdot)\bigr).
  \end{align}
  Then $F_1, F_2$ and $F_3$ are well-defined as mappings from $\mr^r(J;\domLap,L^p)$ into $L^s(J;L^p)$
  and Lipschitzian on bounded sets.
\end{lemma}
\begin{proof}
  We first consider $F_1$ and $F_2$. Taking $\cfo = \bfo$ in Lemma~\ref{l-Lipschitzconty}, we see that
  the operator function in~\eqref{e-mathcalA} belongs to the space
  $C(\overline J;\LL(\dom_{L^p}(\Delta);L^p))$ for every $u \in
  \mr^r(J;\dom_{L^p}(\Delta),L^p)$. Due to the supposition $v_0 \in (L^p,\domLap)_{1-\frac {1}{s},s}$,
  cf.~\eqref{e-anfang} and~\eqref{e-IVreg}, we already know that in fact
  $v_{\mathcal{I}} \in L^s(J;\domLap)$, see Lemma~\ref{l-innisplit}.  For $F_2$, we recall that
  $\sol_1(u)$ belongs to $L^s(J;\domLap)$, cf.\ Theorem~\ref{t-hilfs}. This shows that $F_1$ and $F_2$
  are well-defined.

  Let us prove the Lipschitz properties for $F_{1}$ and $F_2$. For $F_1$, this directly follows from
  Lemma~\ref{l-Lipschitzconty} with $\cfo = \bfo$, and the property $v_{\mathcal{I}} \in
  L^s(J;\domLap)$. On the other hand, $F_2$ is of the form $F_2(u) = \mathcal A_\bfo(u)\sol_1(u)$, where
  $\mathcal A_\bfo$ is the operator in~\eqref{e-mathcalA} for $\cfo = \bfo$, i.e., a product of two
  functions in $u$ which are Lipschitzian and bounded on bounded sets in $\mr^r(J;\domLap,L^p)$ with
  values in the correct spaces, by Lemma~\ref{l-Lipschitzconty} and Corollary~\ref{c-Lipschitzian}, see
  also Remark~\ref{r-uniform}. Hence $F_2$ is also Lipschitzian on bounded sets.

  The assertions on $F_3$ are also satisfied: It remains to collect the continuity of
  $v_{\mathcal{I}},p_{\mathcal{I}}$ and $w_{\mathcal{I}}$ due to Lemma~\ref{l-innisplit} with the
  regularity of $v_0,p_0$ and $w_0$ as in~\eqref{e-IVreg}, the assumptions on $R_1$ (cf.\ Assumption~\ref{a-Koeff}) and the properties of $\sol(\cdot)$ as in Theorem~\ref{t-hilfs} combined with
  Corollary~\ref{c-Lipschitzian}.
\end{proof}

\begin{lemma} \label{l-volt} Define $\mathcal A$ as in~\eqref{e-mathcalA}, there setting $\cfo:= \bfo$.
  Further, put $F:=F_1+F_2+F_3$ as given in~\eqref{e-f1}--\eqref{e-f3}. Then both $\mathcal A$ and $F$
  satisfy the Volterra property, cf.\ Theorem~\ref{t-Amann}.
\end{lemma}

\begin{proof}
  We only need to check the supposition for $\sol$. Since $\sol(u)$ is obtained as the solution of a
  system of semilinear parabolic \emph{forward} equations into which $u$ enters \emph{pointwise} with
  respect to the time variable, it is clear that if $u_1,u_2 \in C(\overline J;C(\overline \Omega))$ with
  $u_1 = u_2$ on a subinterval $I = \iv{0,S} \subseteq J$, then also $\sol(u_1)|_I = \sol(u_2)|_I$. But
  this is exactly the Volterra property.
\end{proof}

Now all suppositions of Theorem~\ref{t-Amann} are proved to be satisfied in order to prove
Theorem~\ref{l-loes0}.

\begin{proofof}{Theorem~\ref{l-loes0}}
  Since we presupposed the correct regularity for the initial value
  $u_0 \in (L^p,\domLap)_{1-\frac {1}{r},r}$, it remains to collect all the assertions from
  Lemmata~\ref{l-Lipschitzconty},~\ref{l-F} and~\ref{l-volt}. With these, Theorem~\ref{t-Amann} is
  applicable and, hence, proves Theorem~\ref{l-loes0}.
\end{proofof}

With Theorem~\ref{l-loes0} at hand, we are now in turn able to prove the main Theorem~\ref{t-mainres} via
Theorem~\ref{l-modisystem}.

\begin{proofof}{Theorem~\ref{l-modisystem}}
  Let $u \in \mr^r(I;\domLap,L^p)$ be the local-in-time solution of~\eqref{e-0u} on an interval
  $I \subset I_\bullet$ as given by Theorem~\ref{l-loes0}. Lemma~\ref{l-embeddd} shows that $u$ admits
  the regularity to obtain $(v,p,w) := \sol(u)$ via Theorem~\ref{t-hilfs}. This proves
  Theorem~\ref{l-modisystem} by construction.
\end{proofof}

\begin{proofof}{Theorem~\ref{t-mainres}}
  We use Theorem~\ref{l-modisystem}.  Let
  \[(u,(\check v,\check p,\check w)) \in \mr^r(I;\domLap,L^p) \times \mr^s_0(I;\domVecLap,\vecL^p)\] be
  the solutions of~\eqref{e-precu}--\eqref{e-precIV} as given by Theorem~\ref{l-modisystem} (we need to
  return to the accented way of denoting the functions, as introduced in~\eqref{e-splitoff}, now). It
  suffices to ``remove'' the cut-off introduced in Definition~\ref{d-2} for
  $(\check v,\check p,\check w)$. Let $M$ be the number from Definition~\ref{d-2} for given $\delta >
  0$. Firstly, from Lemma~\ref{l-innisplit}, we know that
  \[\|v_{\mathcal{I}}\|_{C(\overline I;C(\overline \Omega))} \vee \|p_{\mathcal{I}}\|_{C(\overline
      I;C(\overline \Omega))} \vee \|w_{\mathcal{I}}\|_{C(\overline I;C(\overline \Omega))} \leq M.\] On
  the other hand, since $\check v,\check p$ and $\check w$ are functions from
  $C(\overline I;C(\overline \Omega))$ by Lemma~\ref{l-embeddd} with initial value zero, there exists an
  interval $I_0 = \iv{0,S_0}\ \subseteq I$ such that
  \[\|\check v\|_{C(\overline I_0;C(\overline \Omega))} \vee \|\check p\|_{C(\overline I_0;C(\overline
      \Omega))} \vee \|\check w\|_{C(\overline I_0;C(\overline \Omega))} \leq \frac {\delta}{2}.\] This
  means that
  \begin{multline*}
    R_j^\eta\bigl(u(t),v_{\mathcal{I}}(t) +\check v(t),p_{\mathcal{I}}(t) +\check p(t),w_{\mathcal{I}}(t)
    +\check w(t)\bigr) \\= R_j\bigl(u(t),v_{\mathcal{I}}(t) +\check v(t),p_{\mathcal{I}}(t) +\check
    p(t),w_{\mathcal{I}}(t) +\check w(t)\bigr)
  \end{multline*}
  for every $t \in \overline I_0$, hence $(u,(v,p,w))$ with $(v,p,w)$ as in~\eqref{e-splitoff} are a
  solution to~\eqref{e-preu}--\eqref{e-anfang} on $I_0$, cf.~\eqref{e-inhomanf}. Moreover, $(v,p,w)$
  admit the correct regularity due to
  $(v_{\mathcal{I}},p_{\mathcal{I}},w_{\mathcal{I}}) \in \mr^s(J;\domVecLap,\vecL^p)$, see
  Lemma~\ref{l-innisplit}.
\end{proofof}

\section[Concluding Remarks] {Concluding Remarks}
\label{s-Concluding}

In this concluding chapter we want to comment on possible relaxations and modifications that can
be done to apply our results also to some slightly different situations than those that we have proposed
in the present paper.
\begin{enumerate}[i)]
\item{Reduction to simplified models:} We want to point out again that the simplified
  model~\eqref{simplified} may also be treated by the strategy used above for the full model, with very
  little changes. The same is true for the case of only elliptic equations for $v,p$, and $w$, for which
  one would not need to deal with a nonlocal equation. We refer to the paragraph in the introduction and
  to~\cite{ThermistorPre1}, where such a system was treated.
\item{Regularity of initial data:} We suggest that one can reduce the requirements on the initial values
  considerably, if one is willing and able to work in spaces with temporal weights. The basis of such an
  approach are the results in~\cite{PruessKoenWilk} where it is shown that maximal parabolic regularity
  carries over to spaces with temporal weights. The demanding task would be to prove an analogue of
  Amann's theorem also in this case and, finally, carry out the program of this paper in that
  setting. Clearly, this would be an ambitious program and is completely out of scope here.
\item{Boundary conditions in the model:} Of course, one can also impose other boundary conditions than
  homogeneous Neumann conditions. For example, one can also find references where no-flux boundary
  conditions for the equation of the population density and homogeneous Dirichlet conditions for the
  chemo-attractant, or homogeneous Dirichlet boundary conditions for both equations of the simplified
  system~\eqref{simplified} are considered (see for example~\cite{Diaz} and~\cite{Wolansky}). If still
  other boundary conditions are imposed (as done for instance in~\cite{Myerscough}) or if the
  inhomogeneities $R_i$ consist of more delicate terms such as ones ``living on the boundary''
  $\partial\Omega$, one can proceed in a quite similar way, basing on Assumption~\ref{a-reg} in case of
  pure Dirichlet conditions or mixed boundary conditions. There also exist large classes of domains for
  which the assumption is satisfied in these cases, cf.~\cite{disser}. Then spaces of type $W^{-1,q}$
  would be adequate to considering the system in and the principal functional analytical framework would
  be very similar. In particular, the needed elliptic and parabolic regularity results are also
  available here, cf.~\cite[Ch.~11]{auscher}.
\item{Convex domains:} In contrast to most known results so far we did not assume the domain $\Omega$ to
  be convex. However, if the domain $\Omega$ \emph{is} convex, then it is easier to prove that the
  Keller-Segel system is well-posed: one is enabled to treat the problem in $L^2$, basing on the classical
  result $(-\Delta +1)^{-1} \colon L^2 \to H^2$, cf.~\cite[Ch.~3.2]{grisvard85}. Namely, from this one
  deduces
  \begin{equation*}
    \bigl(L^2,\dom_{L^2}(\Delta )\bigr)_{\theta,1} \embeds \bigl[L^2,\dom_{L^2}(\Delta )\bigr]_\theta \embeds \bigl[L^2,H^2\bigr]_\theta \embeds
    W^{1,4},
  \end{equation*}
  as long as $\theta \ge \frac {1}{2}(1+\frac {d}{4})$, the bound on $\theta$ being strictly smaller than
  $1$ for space dimensions $d = 2$ or $d=3$. Thus, one can principally proceed as in our more general
  proof, thereby avoiding the nontrivial considerations in the non-Hilbert case we used.
\item{Regularity of solutions:} Concerning the equations for $(v,p,w)$, one could choose any other
  integrability index $p \in \iv{\frac q2,\infty}$ for the spatial variable. Moreover, it is
  possible to \emph{bootstrap} the regularity of the solutions by inserting the solutions
  $(v,p,w) \in \mr^s_0(J;\dom_{\mathbb L^p}(\Delta),\vecL^p) \embeds C^\alpha(\overline J;C(\overline
  \Omega))$ of~\eqref{e-precv}--\eqref{e-precw} into the right hand sides, which then each belong to a
  space $C^\beta(J;C(\overline \Omega))$ for some $\beta > 0$.  Now exploiting the fact that $-\Delta$
  also generates an analytic semigroup on $C(\overline \Omega)$ (see~\cite[Rem.~2.6]{Ouhabaz}) and the
  well known results of~\cite[Ch.~4]{luna}, one obtains even more regularity for $(v,p,w)$.
\item{Matrix-valued coefficient functions:} Last, we want to point out a technicality concerning our
  considerations in Chapter~\ref{sec:restriction--nabla} and~\ref{sec:fracpower}. As already mentioned in
  the introduction, these considerations may also be generalized to \emph{real matrix-valued}
  coefficients, that is, the differential operators $-\nabla\cdot\mu\nabla$ where $\mu$ is a bounded
  measurable function on $\Omega$ taking its values in the set of positive definite matrices, since the
  underlying results are available also in this case, cf.~\cite{elschner} and the references therein, see
  also~\cite{disser}. We did not undertake this here because the considered Keller-Segel model is
  restricted to scalar coefficients and the general way to proceed is clear.
\end{enumerate}

\ack

The authors want to thank Herbert Amann (Z\"urich) for valuable discussions. Joachim
Rehberg acknowledges support from ERC grant \#267802: \emph{Analysis of Multiscale Systems Driven by
  Functionals}.


\section*{References}

\begin{thebibliography}{00}




\bibitem{ADN} Agmon,~S., Douglis,~A., Nirenberg,~L.: Estimates near the boundary for solutions of
  elliptic partial differential equations satisfying general boundary conditions I,
  Commun.~Pure~Appl.~Math.~12 (1959) 623--727.

%
\bibitem{Alt:Habil} Alt, W.: Vergleichs\"atze f\"ur quasilineare elliptisch-parabolische Systeme
  partieller Differentialgleichungen, Habilitation, Ruprecht-Karl-Universit\"at Heidelberg, 1980.

\bibitem{Am90} Amann, H.: Dynamic theory of quasilinear parabolic equations. II: Reaction-diffusion
  systems.  Differ. Integral Equ. 3, No.1, (1990) 13-75

\bibitem{AmannDiffEq} Amann, H.: Quasilinear parabolic problems via maximal regularity, Adv.~Differential
  Equations 10 No.~10 (2005) 1081--1110.

\bibitem{AmannRuss} Amann, H.: Non-local quasi-linear parabolic equations, Russ.~Math.~Surv.~60, No.~6,
  1021--1033 (2005); translation from Usp.~Mat.~Nauk~60, No.~6, 21--32 (2005).

\bibitem{A01} Amann, H.: Linear parabolic problems involving measures,
  Rev.~R.~Acad.~Cien.~Serie~A.~Mat.~(RACSAM) 95 (2001) 85--119.

\bibitem{A95} Amann, H.: Linear and quasilinear parabolic problems, Birkh\"auser, Basel, 1995.

\bibitem{Aotani} Aotani, A., Mimura, M., Mollee, T.: A model aided understanding of spot pattern
  formation in chemotactic E. coli colonies, Jpn. J. Ind. Appl. Math. 27 No.~1 (2010) 5–-22.




\bibitem{arendt1} Arendt, W., Chill, R., Fornaro, S., Poupaud, C.: $L^{p}$-maximal regularity for
  nonautonomous evolution equations, J.~Differ.~ Equations 237 No.~1 (2007) 1--26.

\bibitem{auscher} Auscher, P., Badr, N., Haller-Dintelmann, R., Rehberg, J.: The square root problem for
  second-order, divergence form operators with mixed boundary conditions on $L^p$ J.~Evol.~Equ.~15 No.~1
  (2015) 165--208.


\bibitem{bennet} Bennett, C., Sharpley, R.: Interpolation of Operators, Pure and Applied Mathematics,
  Vol.~129, Academic Press, Boston etc., 1988.


\bibitem{bergh} Bergh, J., L\"ofstr\"om, J.: Interpolation spaces. An introduction.  Grundlehren der
  mathematischen Wissenschaften~223, Springer-Verlag, Berlin-Heidelberg-New York, 1976.


\bibitem{Biler:1998} Biler, P.: Local and global solvability of some parabolic system modelling
  chemotaxis, Adv.~Math.~Sci.~Appl.~{8} (1998) 715--743.

\bibitem{BilerZienkiewicz} Biler, P., Zienkiewicz, J.: Existence of solutions for the Keller-Segel model
  of chemotaxis with measures as initial data. (English summary) Bull.~Pol.~Acad.~Sci.~Math.~{63} No.~1
  (2015) 41-–51.

%
\bibitem{Bonner:1967} Bonner, J.~T.: The cellular slime molds, Princeton University Press, Princeton, New
  Jersey, second edition, 1967.
%
\bibitem{Boy:1996} Boy, A.: Analysis for a system of coupled reaction-diffusion parabolic equations
  arising in biology.  Computers Math.~Applic {32} (1996) 15--21.
%
\bibitem{Childress:1981} Childress, S., Percus, J.~K.: Nonlinear aspects of chemotaxis, {Math.~Biosc.}
  {56} (1981) 217--237.
%
\bibitem{cia} Ciarlet, P.~G.: The finite element method for elliptic problems, Studies in Mathematics and
  its Applications, North Holland, Amsterdam/ New York/ Oxford, 1979.

\bibitem{Cie07} Cie{\'{s}}lak, T.: Quasilinear nonuniformly parabolic system modelling chemotaxis,
  J. Math. Anal. Appl. 326 (2007), 1410--1426.

\bibitem{Perthame1} Corrias, L., Perthame, B.: Critical space for the parabolic-parabolic Keller-Segel
  model, Rd.~C.~R.~Math.~Acad.~Sci.~Paris {342} No.~10 (2006) 745-–750.

\bibitem{coulhon} Coulhon, T., Duong, X.~T.: Maximal regularity and kernel bounds: observations on a
  theorem by Hieber and Pr{\"u}ss, Adv.~Differential Equations {5} No.~1--3 (2000) 343--368.

\bibitem{cowl} Cowling, M.~G.: Harmonic analysis on semigroups, Ann.~Math.~{117} No.~2 (1983) 267--283.


\bibitem{Dauge1} Dauge, M.: Neumann and mixed problems on curvilinear polyhedra, Integral Equations
  Oper.~Theory~15 No.~2 (1992) 227--261.

\bibitem{Dauge2} Dauge, M: Problemes de Neumann et de Dirichlet sur un polyedre dans $\R^3$:
  regularit{\'e} dans des espaces de Sobolev $L^p$ (Neumann and Dirichlet problems on a three dimensional
  polyhedron: Regularity in the $L^p$ Sobolev spaces), C.~R.~Acad.~Sci. Paris, Ser.~I 307 No.1 (1988)
  27--32.

\bibitem{Diaz} Diaz, J. I., Nagai, T.: Symmetrization in a parabolic-elliptic system related to
  chemotaxis, Adv. Math. Sci. Appl. 5, No. 2 (1995) 659-–680.

\bibitem{TomKaroJo} Disser, K., ter Elst, A.F.M., Rehberg, J.: H\"older estimates for parabolic operators
  on domains with rough boundary, accepted for Ann. Sc. Norm. Super. Pisa Cl. Sci. (5).

\bibitem{TomKaroJoDiffEq} Disser, K., ter Elst, A.F.M., Rehberg, J.: On maximal parabolic regularity for
  non-autonomous parabolic operators, J. Differential Equations 262 (2017) 2039--2072.

\bibitem{disser} Disser, K., Kaiser, H.-Ch., Rehberg, J.: Optimal Sobolev regularity for linear
  second-order divergence elliptic operators occurring in real-world problems- SIAM J.~Math.~Anal.~47,
  No.~3 (2015) 1719--1746.

\bibitem{dore} Dore, G.: $L^p$ regularity for abstract differential equations, in: Komatsu, Hikosaburo
  (Eds.), Proc.~of the international conference in memory of K.~Yosida, Kyoto University, Japan, 1991,
  Springer, Berlin, Lect.~Notes Math.~1540 (1993) 25--38.



\bibitem{elschner} Elschner, J., Rehberg, J., Schmidt, G.: Optimal regularity for elliptic transmission
  problems including $C^1$ interfaces, Interfaces Free Bound 9 No.~2 (2007), 233--252.

\bibitem{tomjo-consist} ter Elst, A.F.M., Rehberg, J.: Consistent operator semigroups and their
  interpolation, submitted, 2017, eprint arXiv:1703.07126.


\bibitem{elstmeyrrehb} ter Elst, A.F.M., Meyries, M., Rehberg, J.: Parabolic equations with dynamical
  boundary conditions and source terms on interfacesm, Ann.~Mat.~Pura Appl.~(4) 193 No.~5 (2014)
  1295--1318.


\bibitem{ggz} Gajewski, H., Gr\"oger, K., Zacharias, K.: Nichtlineare Operatorgleichungen und
  Operatordifferentialgleichungen, Akademie-Verlag, 1974.

%
\bibitem{Gajewski:1998} Gajewski, H., Zacharias, K.: Global behavior of a reaction-diffusion system
  modelling chemotaxis, Math.~Nachr.~{195} (1998) 77--114.

\bibitem{GiaMarti} Giaquinta, M., Martinazzi, L.: An Introduction to the Regularity Theory for Elliptic
  Systems, Harmonic Maps and Minimal Graphs, Springer Nature, 2012.

\bibitem{gil} Gilbarg, D., Trudinger, N.S.: Elliptic Partial Differential Equations of Second Order,
  2nd. Ed., Springer-Verlag, Berlin, 1983.

\bibitem{ggkr} Griepentrog, J.A., Gr\"oger, K., Kaiser, H.-Ch., Rehberg, J: Interpolation for function
  spaces related to mixed boundary value problems, Math.~Nachr.~{241} (2002) 110--120.

\bibitem{GKR} Griepentrog, J.A., Kaiser, H.-Ch., Rehberg, J.: Heat kernel and resolvent properties for
  second order elliptic differential operators with general boundary conditions on $L^p$ .
  Adv.~Math.~Sci.~Appl.~11 No.1 (2001) 87--112.

\bibitem{grisvard85} Grisvard, P.: Elliptic problems in nonsmooth domains, Pitman, Boston, 1985.

\bibitem{groeger89} Gr\"oger, K.: A {$W^{1,p}$}--estimate for solutions to mixed boundary value problems
  for second order elliptic differential equations, Math.~Ann.~ {283} (1989) 679--687.






\bibitem{Ziegler} Haller-Dintelmann, R., H\"oppner, W., Kaiser, H.-Ch., Rehberg, J., Ziegler, G.M.:
  Optimal elliptic Sobolev regularity near three-dimensional multi-material Neumann vertices,
  Funct.~Anal.~Appl.~48 No.~3 (2014) 208--222; translation from Funkts.~ Anal Prilozh.~48 No.~3 (2014)
  63--83.

\bibitem{HKR} Haller-Dintelmann, R., Kaiser, H.-Ch., Rehberg, J.: Elliptic model problems including mixed
  boundary conditions and material heterogeneities, J.~Math.~Pures Appl.~(9) 89 No.~1 (2008) 25--48.
%
\bibitem{Herrero1} Herrero, M.~A., Vel\'azquez, J.~J.~L.: A blow-up mechanism for a chemotaxis model,
  Ann. Sc. Norm. Super. Pisa Cl. Sci. {24} (1997) 633--683.



\bibitem{HiebPruess} Hieber, M., Pr\"uss, J.: Heat kernels and maximal $L^p -L^q$ estimates for parabolic
  evolution equations, Commun.~Partial Differ.~Equations 22 No.~9-10 (1997) 1647--1669.

\bibitem{HiebRehb} Hieber, M., Rehberg, J.: Quasilinear parabolic systems with mixed boundary conditions
  on nonsmooth domains.,SIAM J.~Math.~Anal.~40 No.~1 (2008) 292--305.

\bibitem{Hillen:Painter:2002} Hillen, T, Painter, K.: Volume-filling and quorum-sensing in models for
  chemotaxis movement, Canad.~Appl.~Math.~Quart.~{10} (2002) 501--543.

\bibitem{Hillen:Painter} Hillen, T, Painter, K.: A user's guide to PDE models for chemotaxis,
  J.~Math.~Biol.~{58} (2009) 183--217.
%

\bibitem{Horstmann4} Horstmann, D.: The nonsymmetric case of the Keller-Segel model in chemotaxis: some
  recent results, Nonlinear Differ.~Equ.~Appl.~{8} (2001) 399--423.
%
\bibitem{Horstmann5} Horstmann, D., Wang, G.: Blow-up in a chemotaxis model without symmetry assumptions,
  Eur.~J.~Appl.~Math.~{12} (2001) 159--177.
%
\bibitem{Horstmann6} Horstmann, D.: On the existence of radially symmetric blow-up solutions for the
  Keller-Segel model, J.~Math.~Biol.~{44} (2002) 463--478.
%
\bibitem{Horstmann1} Horstmann, D.: From 1970 until present: The Keller-Segel model in chemotaxis and its
  consequences I, Jahresber.~Deutsch.~Math.-Verein.~105 No.~3 (2003) 103--165.

\bibitem{Horstmann2} Horstmann, D.: From 1970 until present: The Keller-Segel model in chemotaxis and its
  consequences II, Jahresber.~Deutsch.~Math.-Verein.~106 No.~2 (2004) 51--69.
%
\bibitem{Horstmann9} Horstmann, D., Stevens, A.: A constructive approach to traveling waves in
  chemotaxis, J.~Nonlinear Sci.~{14} (2004) 1--25.

\bibitem{Horstmann7} Horstmann, D., Winkler, M.: Boundedness vs.\ blow-up in a chemotaxis system.
  J.~Differential Equations {215} No.~1 (2005) 52-–107.

\bibitem{Horstmann3} Horstmann, D.: Generalizing Keller-Segel: Lyapunov functionals, steady state
  analysis and blow-up results for multi-species chemotaxis models in the presence of attraction and
  repulsion between competitive interacting species, J.~Nonlinear~Sci.~{21} (2011) 231--270.
%
\bibitem{Horstmann8} Horstmann, D., Strehl, R., Sokolov, A., Kuzmin, D., Turek, S.: A
  positivity-preserving finite element method for chemotaxis problems in 3D, J.~Comput.~Appl.~Math {239}
  (2013) 290-–303.

\bibitem{jer/ke} Jerison, D., Kenig, C.: The inhomogeneous Dirichlet problem in Lipschitz domains,
  J.~Funct.~Anal.~{130} No.~1 (1995) 161--219.

\bibitem{Kang1} Kang, K., Stevens, A.: Blowup and global solutions in a chemotaxis-growth
  system. Nonlinear Anal.~{135} (2016) 57--72.

\bibitem{KellerSegel} Keller, E.\ F., Segel, L.\ A.: Initiation of slime mold aggregation viewed as an
  instability, J.~Theor.~Biol.~{26} No.~3 (1970) 399--415.

\bibitem{katosr} Kato, T.: Fractional powers of dissipative operators, J.~Math.~Soc.~Japan Vol.~13, No.~3
  (1961) 246--274.

\bibitem{kato} Kato, T.: Perturbation theory for linear operators, Grundlehren der mathematischen
  Wissenschaften,~{132}, Springer Verlag, Berlin, 1984.

\bibitem{PruessKoenWilk} K\"ohne, M., Pr\"uss, J., Wilke, M.: On quasilinear parabolic evolution
  equations in weighted $L^p$ -spaces, J.~Evol.~Equ.~{10} No.~2 (2010) 443--463.

\bibitem{ladypara} Ladyzhenskaya, O.A., Solonnikov, V.A., Ural{\textquoteright}tseva, N.N.: Linear and
  quasilinear equations of parabolic type, American Mathematical Society, Providence, R.I., 1968.

\bibitem{lambert} Lamberton,~D.: Equations d'{\'e}volution lin{\'e}aires associ{\'e}es {\`a} des
  semi-groupes de contractions dans les espaces $L^p$., J.~Funct.~Anal.~{72} (1987) 252--262

\bibitem{lang} Lang, S.: Real and functional analysis, 3.\ ed., Graduate Texts in Mathematics 142,
  Springer-Verlag, New York, 1993.


\bibitem{luna} Lunardi, A.: Analytic semigroups and optimal regularity in parabolic problems.  Modern
  Birkh\"auser Classics. Birkh\"auser/Springer Basel AG, Basel, 1995

\bibitem{mazyasob} Maz'ya, V.: Sobolev spaces, Springer, 1985

\bibitem{archive} Maz'ya, V., Elschner, J., Rehberg, J., Schmidt, G.: Solutions for quasilinear nonsmooth
  evolution systems in $L^p$ Arch.~Ration.~Mech.~Anal.~{171} No.~2 (2004) 219-262.

\bibitem{ThermistorPre1} Meinlschmidt, H., Meyer, C., Rehberg, J.: Optimal control of the thermistor
  problem Part 1: Existence of optimal controls, submitted.

\bibitem{mercier} Mercier, D.: Minimal regularity of the solutions of some transmission problems,
  Math.~Meth.~Appl.~Sci., {26} (2003) 321--348.

\bibitem{morr} Morrey, C.~B.~jun.: Multiple integrals in the calculus of variations, Grundlehren der
  mathematischen Wissenschaften~130, Springer-Verlag, Berlin-Heidelberg-New York, 1966.

\bibitem{Myerscough} Myerscough, M. R., Maini, P. K., Painter, K. J.: Pattern Formation in a generalized
  chemotaxis model, Bulletin of Mathematical Biology 60 (1998) 1--26.



%

\bibitem{Nagai:1997} Nagai, T., Senba, T., Yoshida, K.: Application of the {Moser-Trudinger} inequality
  to a parabolic system of chemotaxis, {Funkcial.~Ekvac. Ser.~Int.} {40} (1997) 411--433.

\bibitem{Nagai1} Nagai, T.: Behavior of solutions to a parabolic-elliptic system modelling chemotaxis,
  J.~Korean Math.~Soc.~{37} (2000) 721--733.

\bibitem{Nagel89} Nagel, R.: Operator matrices and reaction-diffusion systems,
  Seminario~Mat.~e.~Fis.~di~Milano (1989) 59--185.

%
\bibitem{Nanjundiah:1992} Nanjundiah, V., Shweta, S.: The determination of spatial pattern in
  {Dictyostelium} {discoideum}, J.~Biosci.~{17} (1992) 353--394.
%
%
\bibitem{Osaki:2001} Osaki, K., Yagi, A.: Finite dimensional attractors for one-dimensional Keller-Segel
  equations, Funkcial.~Ekvac.~{44} (2001) 441--469.
%
\bibitem{Osaki:2002} Osaki, K., Yagi, A.: Global existence for a chemotaxis-growth system in
  $\R^2$. Adv.~Math.~Sci.~Appl.  {12} no.~2 (2002) 587--606.
%
\bibitem{Ouhabaz} Ouhabaz, E.-M.: Gaussian estimates and holomorphy of semigroups,
  Proc.~Am.~Math.~Soc.~123 No.5 (1995) 1465--1474.
%
%
\bibitem{Ouh05} Ouhabaz, E.: Analysis of Heat Equations on domains, Vol.~31 of London Mathematical
  Society Monographs Series, Princeton University Press, Princeton, 2005.

\bibitem{pazy} Pazy, A.: Semigroups of Linear Operators and Applications to Partial Differential
  Equations, Springer, Berlin, 1983.

\bibitem{Pierre10} Pierre, M.: Global Existence in Reaction-Diffusion Systems with Control of Mass: a
  Survey, Milan~J.~Math. (2010) 78--417.

\bibitem{pruess} Pr\"uss, J.: Maximal regularity for evolution equations in $L^p$-spaces,
  Conf.~Semin.~Mat.~Univ.~Bari, 285 (2002) 1--39.




\bibitem{schnaubelt} Pr\"uss, J., Schnaubelt, R.: Solvability and maximal regularity of parabolic
  evolution equations with coefficients continuous in time, J.~Math.~Anal.~Appl.~256 No.2 (2001)
  405--430.


%
\bibitem{Senba:Suzuki:Applied-Analysis2004} Senba T., Suzuki, T.: Applied Analysis, Imperial College
  Press, 2004.

\bibitem{Stinner1} Stinner, Ch., Winkler, M.: Global weak solutions in a chemotaxis system with large
  singular sensitivity, Nonlinear Anal.~Real World Appl.~{12} No.  6 (2011) 3727--3740.

\bibitem{taowinkler} Tao, Y., Winkler, M.: Boundedness in a quasilinear parabolic-parabolic Keller-Segel
  system with subcritical sensitivity, J. Diff. Eq. 252 No. 1, (2012) 692-715


\bibitem{Tello1} Tello, J.~I., Winkler, M.: A chemotaxis system with logistic source, Comm.~Partial
  Differential Equations {32} No.~4-6 (2007) 849--877.

\bibitem{triebel} Triebel, H.: Interpolation theory, function spaces, differential operators, North
  Holland Publishing Company, 1978.

\bibitem{Vasiev-Panfilov} Vasiev, B. N., Hogeweg, P., Panfilov, A. V.: Simulation of Dictyostelium
  discoideum Aggregation via Reaction-Diffusion Model, Physical Review Letters 73 No. 23 (1994)
  3173--3176.

\bibitem{Wang} Wang, Z.: An Mathematics of traveling waves in chemotaxis - review paper, Discrete
  Contin.~Dyn.~Syst.~Ser.~B {18} No.~3 (2013) 601--641.

\bibitem{Winkler2} Winkler, M.: Global solutions in a fully parabolic chemotaxis system with singular
  sensitivity, Math.~Methods Appl.~Sci.~{34} No.~2 (2011) 176--190.

\bibitem{Winkler1} Winkler, M.: Finite-time blow-up in the higher-dimensional parabolic-parabolic
  Keller-Segel system, J.~Math.~Pures Appl.~(9) 100 No.~5 (2013) 748--767.

\bibitem{Wolansky} Wolansky, G.: A critical parabolic estimate and application to nonlocal equations
  arising in chemotaxis, Appl. Anal. 66 No. 3-4 (1997) 291-–321.

\bibitem{Wrz04} Wroszek, D.: Global attractor for a chemotaxis model with prevention of overcrowding,
  Nonlinear~Anal.~Theory~Methods Appl.~59 (2004), 1293--1310.


\bibitem{Xiang1} Xiang, T.: Boundedness and global existence in the higher-dimensional
  parabolic-parabolic chemotaxis system with/without growth source, J.~Differential Equations {258}
  No.~12 (2015) 4275--4323.

\bibitem{Yagi} Yagi, A.: {Norm} behavior of solutions to a parabolic system of chemotaxis, Math.~Japonica
  {45} (1997) 241--265.

\bibitem{Zanger} Zanger, D.: The inhomogeneous Neumann problem in Lipschitz domains, Commun.~Partial
  Differ.~Equations {25} No.9-10 (2000) 1771--1808.

\bibitem{Zhang} Zhang, K., On coercivity and regularity for linear elliptic systems, Calc. Var.  Partial
  Differ. Equ. 40 No. 1 (2011) 65--97.

\end{thebibliography}
\end{document}